\documentclass{amsart}
\usepackage{amsfonts}
\usepackage{amsmath,amssymb}
\usepackage{amsthm}
\usepackage{amscd}
\usepackage{graphics}
\usepackage{graphicx}

\theoremstyle{remark}{
\newtheorem{Def}{{\rm Definition}}
\newtheorem{Ex}{{\rm Example}}
\newtheorem{Rem}{{\rm Remark}}
\newtheorem{Prob}{{\rm Problem}}

}
\theoremstyle{plain}
{

\newtheorem{Prop}{Proposition}
\newtheorem{Thm}{Theorem}
\newtheorem{MainThm}{Main Theorem}

}

\begin{document}
\title[Reconstructing real algebraic maps locally like moment-maps]{Reconstructing real algebraic maps locally like moment-maps with prescribed images and compositions with the canonical projections to the $1$-dimensional real affine space}
\author{Naoki kitazawa}

\address{Institute of Mathematics for Industry, Kyushu University, 744 Motooka, Nishi-ku Fukuoka 819-0395, Japan\\
 TEL (Office): +81-92-802-4402 \\
 FAX (Office): +81-92-802-4405 \\
}
\email{n-kitazawa@imi.kyushu-u.ac.jp, naokikitazawa.formath@gmail.com}
\urladdr{https://naokikitazawa.github.io/NaokiKitazawa.html}
\maketitle
\begin{abstract}
We present new real algebraic maps of non-positive codimensions with prescribed images whose boundaries consist of explicit non-singular real algebraic hypersurfaces satisfying so-called "transversality". Explicit information on important real polynomials is also given. Their preimages are one-point sets or products of spheres. They are locally like so-called {\it moment} maps. 

Celebrated theory of Nash and Tognoli says that smooth closed manifolds are {\it non-singular} real algebraic manifolds and the zero sets of some real polynomial maps. In general, we can approximate smooth functions or more generally, maps, by real algebraic ones. It is in general difficult to have explicit examples. We have constructed maps of a specific class of the present class containing the canonical projections of the unit spheres previously where preimages are one-point sets or spheres.  

We also present explicit families of functions represented as compositions of such maps with the canonical projections.

\end{abstract}
\section{Introduction.}
\label{sec:1}

Real algebraic geometry has a long nice history and studies ({\it non-singular}) real algebraic manifolds or more generally, varieties. For related topics, see \cite{bochnakcosteroy, bochnakkucharz, kollar, kucharz, nash, tognoli} for example. Here, we concentrate on non-singular real algebraic manifolds.

Some of Nash and Tognoli's celebrated theory shows that smooth closed manifolds are non-singular real algebraic manifolds and the zero sets of some real polynomial maps. In considerable cases, we can approximate smooth functions on {\it non-singular} real algebraic manifolds or more generally, maps between non-singular real algebraic manifolds, by real algebraic ones. Here, as another problem, we study construction of explicit real algebraic maps of non-positive codimensions and their explicit global information. Some explicit examples of such functions and maps are well-known. The canonical projections of spheres embedded naturally in Euclidean spaces are simplest ones. Some natural functions on so-called projective spaces, Lie groups and their quotient spaces are also well-known. They are represented as real polynomials. 
Knowing explicit global information of the maps and the manifolds such as topological properties is difficult in general.
See \cite{maciasvirgospereirasaez, ramanujam, takeuchi} for example.

We give some new answers to this problem explicitly. 
We construct real algebraic maps whose codimensions are non-positive with prescribed images and preimages. The boundaries of the images consist of non-singular real algebraic hypersurfaces intersecting with the condition on "transversality". In addition, preimages are products of spheres. We can also know real polynomials for the desired zero sets and the manifolds explicitly. Such maps also generalize the canonical projections of spheres and locally like so-called {\it moment maps}.
Our method applies some singularity theory of smooth maps, some theory on differential topology of manifolds, and some elementary real algebraic arguments and is interdisciplinary. Our study also adds to our related pioneering studies \cite{kitazawa3, kitazawa7}. Studies on domains formed by real algebraic curves \cite{bodinpopescupampusorea, kohnpieneranestadrydellshapirosinnsoreatelen, 
	sorea1, sorea2} also motivate us.
We also investigate functions represented as compositions of our new maps with the canonical projections. They are shown to be {\it Morse-Bott}. We mainly study global topological descriptions by the {\it Reeb graphs}, the natural spaces of connected components of preimages.
\subsection{Smooth manifolds and maps.}
Let $X$ be a topological space having the structure of some cell complex whose dimension is finite. We can define the dimension $\dim X$ uniquely. The dimension $\dim X$ is an integer of course. 
A topological manifold is well-known to have the structure of a CW complex. A smooth manifold is well-known to have the structure of a polyhedron. We can define the structure of a certain polyhedron for a smooth manifold canonically. This is a so-called PL manifold. It is also well-known that a topological space having the structure of a polyhedron whose dimension is at most $2$ has the structure of a polyhedron uniquely. For a topological manifold of dimension at most $3$, the same fact holds. This respects celebrated and well-known theory by \cite{moise} for example. Let ${\rm Int}\ X$ denote the interior of  a manifold $X$ and $\partial X$ the boundary $X-{\rm Int}\ X$ of $X$.
  
Let ${\mathbb{R}}^k$ denote the $k$-dimensional Euclidean space. It is a simplest $k$-dimensional smooth manifold and the Riemannian manifold with the standard Euclidean metric. 
Let $\mathbb{R}:={\mathbb{R}}^1$ and $\mathbb{N} \subset \mathbb{Z} \subset \mathbb{R}$ be the set of all positive integers and that of all integers respectively.
For each point $x \in {\mathbb{R}}^k$, we can define $||x|| \geq 0$ as the distance between $x$ and the origin $0$ under this metric.
This is also naturally a simplest real algebraic manifold: the $k$-dimensional real affine space. Let $S^k:=\{x \in {\mathbb{R}}^{k+1} \mid ||x||=1\}$ denote the $k$-dimensional unit sphere, which is a $k$-dimensional smooth compact submanifold of ${\mathbb{R}}^{k+1}$ and has no boundary. It is connected for any positive integer $k \geq 1$. It is a discrete two-point set for $k=0$. It is the zero set of the real polynomial $||x||^2-1={\Sigma}_{j=1}^{k+1} {x_j}^2-1$ with  $x:=(x_1,\cdots,x_{k+1})$ and a real algebraic (sub)manifold (hypersurface). 
Let $D^k:=\{x \in {\mathbb{R}}^{k} \mid ||x|| \leq 1\}$ denote the $k$-dimensional unit disk. It is a $k$-dimensional smooth compact and connected submanifold of ${\mathbb{R}}^{k}$ for any positive integer $k \geq 1$. Of course we have $\partial D^k=S^{k-1}$.

For a smooth manifold $X$, let $T_xX$ denote the tangent vector space at $x \in X$. Let $c:X \rightarrow Y$ be a differentiable map from a differentiable manifold $X$ into another one $Y$. Let ${dc}_x:T_x X \rightarrow T_{c(x)}Y$ denote the differential of $c$ at $x \in X$: this is a linear map between the tangent vector spaces. If the rank of the differential ${dc}_x$ is smaller than the minimum of the (multi)set $\{\dim X, \dim Y\}$, then $x$ is a {\it singular point} of $c$.  For a {\it singular} point $x \in X$ of $c$, the value $c(x)$ is a {\it singular value} of $c$. Let $S(c)$ denote the set of all singular points of $c$. 

In our paper, we consider smooth maps, defined as maps of the class $C^{\infty}$, as differentiable maps, unless otherwise stated. A canonical projection of the Euclidean space ${\mathbb{R}}^k$ is denoted
by ${\pi}_{k,k_1}:{\mathbb{R}}^{k} \rightarrow {\mathbb{R}}^{k_1}$. This is defined as the map mapping 
each point $x=(x_1,x_2) \in {\mathbb{R}}^{k_1} \times {\mathbb{R}}^{k_2}={\mathbb{R}}^k$ to the first component $x_1 \in {\mathbb{R}}^{k_1}$ with the conditions on the dimensions given by $k_1, k_2>0$ and $k=k_1+k_2$. A canonical projection of the unit sphere $S^{k-1}$ is the restriction of it.

Here, a real algebraic manifold is represented as a union of connected components of the intersection of finitely many zero sets of real polynomials. {\it Non-singular} real algebraic manifolds are defined naturally by the implicit function theorem for the polynomials or the maps defined canonically from the polynomials. The real affine space and the unit sphere are simplest non-singular ones. {\it Real algebraic} maps here are represented as the compositions of the canonical embeddings into the real affine spaces with canonical projections.

\subsection{Our main result.}


Hereafter, let ${\mathbb{N}}_a$ be the set of all elements of $\mathbb{N}$ smaller than or equal to a given real number $a$. For example, for a positive integer $i$, ${\mathbb{N}}_i:=\{1,\cdots,i\}$ and the set of all integers from $1$ to $i$.

\begin{MainThm}
	\label{mthm:1}
	Let $l_1$, $l_2$ and $n$ be positive integers.
	Let $\{S_j\}_{j=1}^{l_1}$ be a family of non-singular
	real algebraic hypersurfaces of ${\mathbb{R}}^n$ each $S_j$ of which is connected and the zero set of a real polynomial $f_j$.
Let $D \neq \emptyset$ be a connected open subset of ${\mathbb{R}}^n$ whose closure $\overline{D}$ is compact and connected. Suppose also the following.
	\begin{enumerate}
		\item \label{mthm:1.1} 
For any sufficiently small open neighborhood $U_D$ of $\overline{D}$, it holds that $D=U_D \bigcap {\bigcap}_{j=1}^{l_1} \{x \mid f_j(x)>0\}$. 

\item \label{mthm:1.2} For each increasing sequence $\{j_i\}_{i=1}^{l_3} \subset {\mathbb{N}}_{l_1}$ {\rm (}$1 \leq l_3 \leq l_1${\rm )} and each point $p$ in the intersection ${\bigcap}_{i=1}^{l_3} S_{j_i} \bigcap \overline{D}$, the dimension of the intersection ${\bigcap}_{i=1}^{l_3} T_{p} S_{j_i}$ of the tangent vector spaces is $n-l_3$. 
\item \label{mthm:1.3} A surjective map $m_{l_1,l_2}:{\mathbb{N}}_{l_1} \rightarrow {\mathbb{N}}_{l_2}$ enjoying the following properties exists: for any increasing sequence $\{j_i\}_{i=1}^{l_3}$ making the set ${\bigcap}_{i=1}^{l_3} S_{j_i} \bigcap \overline{D}$ non-empty,  the restriction of $m_{l_1,l_2}$ to the set $\{j_i\}_{i=1}^{l_3}$ is always injective.
	\end{enumerate}
Let $m_{l_2}:{\mathbb{N}}_{l_2} \rightarrow \mathbb{N} \sqcup \{0\}$ be a map. Let $m:=n+{\Sigma}_{j=1}^{l_2} m_{l_2}(j)$. Then there exist an $m$-dimensional non-singular real algebraic closed manifold $M$ and a smooth real algebraic map $f:M \rightarrow {\mathbb{R}}^n$ with the image $f(M)$ being the closure $\overline{D}$, the set $f^{-1}(p)$ {\rm (}$p \in \overline{D}${\rm )} being a one-point set or diffeomorphic to a manifold of the form ${\prod}_{j=1}^{l_k} S^{k_j}$ and the property 
$\dim f^{-1}(p) \leq m-n$  {\rm (}$p \in \overline{D}${\rm )}. In the case the function $m_{l_2}$ is always positive for example, our manifold $M$ can be obtained as a connected one. 
	\end{MainThm}

This is an extension or a variant of main results of the author, presented in \cite{kitazawa3, kitazawa7}. There the hypersurfaces do not intersect. See also \cite{kitazawa4, kitazawa9}. Additional Main Theorems, presented in the third section, are on the functions represented as the compositions of the maps with the canonical projections and their global structures. \cite{kitazawa3} also studies this kind of problems. More precisely, these new theorems show explicit functions  represented as the compositions of maps into ${\mathbb{R}}^2$ of Main Theorem \ref{mthm:1} with each $S_j$ being a circle with the canonical projection ${\pi}_{2,1}$.

We add exposition on Main Theorem \ref{mthm:1} (\ref{mthm:1.2}). This is on "transversality". From the condition on the tangent vector spaces, for any open set $U_D$ as in (\ref{mthm:1.1}), the subset $U_D \bigcap {\bigcap}_{i=1}^{l_3} S_{j_i}$ is an {\rm (}$n-l_3${\rm )}-dimensional smooth regular submanifold of $U_D \subset {\mathbb{R}}^n$ with no boundary and we have ${\bigcap}_{i=1}^{l_3} T_{p} S_{j_i}=T_p (U_D \bigcap {\bigcap}_{i=1}^{l_3} S_{j_i})$  {\rm (}if the subset is not empty{\rm )}. This subset is also a closed subset of $U_D \subset {\mathbb{R}}^n$. See \cite{golubitskyguillemin} for transversality.

The next section is on Main Theorem \ref{mthm:1}.
Main Theorems \ref{mthm:2} and \ref{mthm:3} are presented as an application in the third section. The fourth section presents additional remarks. \\
\ \\
\noindent {\bf Conflict of interest.} \\
The author was a member of the project JSPS Grant Number JP17H06128.
The author was a member of the project JSPS KAKENHI Grant Number JP22K18267. Principal Investigator for them is all Osamu Saeki.  The author works at Institute of Mathematics for Industry (https://www.jgmi.kyushu-u.ac.jp/en/about/young-mentors/) and this is closely related to our study. Our study thanks them for their supports. The author is also a researcher at Osaka Central
Advanced Mathematical Institute (OCAMI researcher), supported by MEXT Promotion of Distinctive Joint Research Center Program JPMXP0723833165. He is not employed there. This is for our studies
and our study also thanks this. \\
\ \\
{\bf Data availability.} \\
Data essentially supporting our present study are all in the
 paper. We also note that the present version is a revised version of the previous version \cite{kitazawa8}. We have revised based on our consideration. This does not affect our main ingredients. We have added some observations and exposition. We also have changed the title.   
\section{On Main Theorem \ref{mthm:1}.}
\subsection{Our proof of Main Theorem \ref{mthm:1}.}

Hereafter, we use "$\prod$" for the products of numbers, functions or sets. We use the notation in the forms ${\prod}_{j=1}^l$ for a positive integer $l$ and ${\prod}_{x \in X}$ for a (finite) set $X$ for example. For coordinates and points for example, we use the notation like $x=(x_1,\cdots,x_l)$ with a positive integer $l$ for example.

\begin{Thm}\label{thm:1}
In Main Theorem \ref{mthm:1}, we can also construct a suitable map $f:M \rightarrow {\mathbb{R}}^n$ on a suitable manifold $M$ in such a way that each preimage $f^{-1}(p)$ is as follows.
\begin{enumerate}
\setcounter{enumi}{3}
	\item For $p \in D$, it is diffeomorphic to ${\prod}_{j=1}^{l_2} S^{m_{l_2}(j)}$.
	\item Let $\{j_i\}_{i=1}^{l_3} \subset {\mathbb{N}}_l$ be an increasing sequence of integers {\rm (}$1 \leq l_3 \leq l_1${\rm )}. Let $p$ be a point of the set ${\bigcap}_{i=1}^{l_3} S_{j_i} \bigcap \overline{D}$ such that for any increasing sequence $\{{j^{\prime}}_i\}_{i=1}^{l_3+1}$ containing the increasing sequence $\{j_i\}_{i=1}^{l_3}$ as a subsequence, $p \notin {\bigcap}_{i=1}^{l_3+1} S_{{j^{\prime}}_i}$ holds. Then it is diffeomorphic to ${\prod}_{j \in {\mathbb{N}}_{l_2}-\{m_{l_1,l_2}(j_i)\}_{i=1}^{l_3}} S^{m_{l_2}(j)}$ if ${\mathbb{N}}_{l_2}-\{m_{l_1,l_2}(j_i)\}_{i=1}^{l_3}$ is not empty and the one-point set if ${\mathbb{N}}_{l_2}-\{m_{l_1,l_2}(j_i)\}_{i=1}^{l_3}$ is empty.
\end{enumerate}
\end{Thm}
This is an extension of a main result of the author of \cite{kitazawa3}. 
There hypersurfaces do not intersect.
The result is also presented as a part of Theorem \ref{thm:1} of \cite{kitazawa4}. Some essential steps are same as those of their original proofs. 
\begin{proof}[A proof of Main Theorem \ref{mthm:1} with Theorem \ref{thm:1}]

We define a subset $S_0:=\{(x,y_1,\cdots,y_{l_2}) \in U_D \times {\prod}_{i=1}^{l_2} {\mathbb{R}}^{m_{l_2}(i)+1} \subset {\mathbb{R}}^n \times {\prod}_{i=1}^{l_2} {\mathbb{R}}^{m_{l_2}(i)+1}={\mathbb{R}}^{m+l_2} \mid {\prod}_{j \in {m_{l_1,l_2}}^{-1}(i)} (f_{j}(x_1,\cdots,x_n))-{||y_i||}^2=0, i \in {\mathbb{N}}_{l_2}\} \subset {\mathbb{R}}^{m+l_2}$. Here $y_i:=(y_{i,1},\cdots,y_{i,m_{l_2}(i)+1})$. 
\\
\ \\
STEP 0 The subset $S_0$ is not empty. \\
The connected open set $D$ is assumed to be non-empty.
We remember the condition (\ref{mthm:1.3}) and the map  $m_{l_1,l_2}:{\mathbb{N}}_{l_1} \rightarrow {\mathbb{N}}_{l_2}$.

The set ${m_{l_1,l_2}}^{-1}(i)$ is not empty for any $i \in {\mathbb{N}}_{l_2}$ since $m_{l_1,l_2}$ is assumed to be surjective. 

They imply that the set $S_0$ is not empty. \\

\ \\
STEP 1 A proof of the fact by implicit function theorem that the set $S_0$ is a smooth compact submanifold with no boundary. \\
We calculate the partial derivatives of the real polynomial function defined by the real polynomial ${\prod}_{j \in {m_{l_1,l_2}}^{-1}(i)}  (f_{j}(x_1,\cdots,x_n))-||y_i||^2={\prod}_{j \in {m_{l_1,l_2}}^{-1}(i)}  (f_{j}(x_1,\cdots,x_n))-{\Sigma}_{j^{\prime}=1}^{m_{l_2}(i)+1} {y_{i,j^{\prime}}}^2$ for variables $x_j$ and $y_{i,j^{\prime}}$ ($1 \leq i \leq l_2$). In STEP 1-1 and STEP 1-2, we apply implicit function theorem around each point of $S_0$. After that, in STEP 1-3, we give additional arguments to complete our proof of the fact that the set $S_0$ is a smooth compact submanifold with no boundary. \\
\ \\
STEP 1-1 The case for a point $(x_0,y_0) \in S_0 \bigcap (\overline{D} \times {\mathbb{R}}^{m-n+l_2}) \subset U_D \times {\mathbb{R}}^{m-n+l_2}$ such that for $y_0:=(y_{0,1},\cdots,y_{0,l_2})$, $y_{0,j}$ is not the origin for any $1 \leq j \leq l_2$. \\
\ \\
We have the inequality $f_{j}(x_0)>0$ for any $1 \leq j \leq l_1$. This is thanks to our first condition (\ref{mthm:1.1}). As a result we also have the inequality ${\prod}_{j \in {m_{l_1,l_2}}^{-1}(i)}  f_{j}(x_0)>0$ for each integer $1 \leq i \leq l_2$. 
We have a real number $2y_{i,{j_0}^{\prime}}=2y_{0,i,{j_0}^{\prime}} \neq 0$ as the value of the partial derivative at the point $(x_0,y_0)$ for some variable $y_{i,{j_0}^{\prime}}$: here, $y_{0,i,{j_0}^{\prime}}$ is from $y_{0,i}:=(y_{0,i,1},\cdots,y_{0,i,m_{l_2}(i)+1})$ with $y_0:=(y_{0,1},\cdots,y_{0,l_2})$.
Let us abuse the same notation. We also have $0$ as the value of the partial derivative at the point for any variable $y_{i^{\prime},{j}^{\prime}}$ satisfying $i^{\prime} \neq i$.

The differential of the restriction of the map into ${\mathbb{R}}^{l_2}$ defined canonically from the $l_2>0$ real polynomials ${\prod}_{j \in {m_{l_1,l_2}}^{-1}(i)}  (f_{j}(x))-||y_i||^2$ ($1 \leq i \leq l_2$) at $(x_0,y_0)$ is of rank $l_2$. 
The point is not a singular point of the map. Moreover, around the point $(x_0,y_0)$, the set $S_0$ is represented as the graph of a smooth map as follows. The variable of the target of the map consists of exactly $l_2$ components $y_{i,{j_0}^{\prime}}$ ($1 \leq i \leq l_2$). 
\\
\ \\
STEP 1-2 The case for a point $(x_{\rm O},y_{\rm O}) \in S_0 \bigcap (\overline{D} \times {\mathbb{R}}^{m-n+l_2}) \subset U_D \times {\mathbb{R}}^{m-n+l_2}$ such that for $y_{\rm O}:=(y_{{\rm O},1},\cdots,y_{{\rm O},l_2}) \in {\mathbb{R}}^{m-n+l_2}$, $y_{{\rm O},j}$ is the origin for some $1 \leq j \leq l_2$. \\
\ \\
By the condition (\ref{mthm:1.1}), we can also assume that the fact $x_{\rm O} \in {\bigcap}_{i=1}^{l_3} S_{j_i} \bigcap \overline{D}$ holds for some increasing sequence $\{j_i\}_{i=1}^{l_3}$ and that the fact $x_{\rm O} \notin {\bigcap}_{i=1}^{l_3+1} S_{{j^{\prime}}_i}$ holds for any increasing sequence $\{{j^{\prime}}_i\}_{i=1}^{l_3+1}$ containing the sequence $\{j_i\}_{i=1}^{l_3}$ as a subsequence.

We define a set of all points 
${x_{\rm O}}^{\prime} \in {\bigcap}_{i=1}^{l_3} S_{j_i} $ such that the fact ${x_{\rm O}}^{\prime} \notin {\bigcap}_{i=1}^{l_3+1} S_{{j^{\prime}}_i}$ holds for any increasing sequence $\{{j^{\prime}}_i\}_{i=1}^{l_3+1}$ containing the sequence $\{j_i\}_{i=1}^{l_3}$ as a subsequence.
Let $S_{0,\{j_i\}_{i=1}^{l_3}}$ denote the set. It is a subset of the ($n-l_3$)-dimensional manifold ${\bigcap}_{i=1}^{l_3} S_{j_i}$ of course.

By the condition (\ref{mthm:1.1}),
for each point ${x_{\rm O}}^{\prime} \in S_{0,\{j_i\}_{i=1}^{l_3}} \bigcap \overline{D}$ and a sufficiently small open neighborhood $U_{{x_{\rm O}}^{\prime},{\bigcap}_{i=1}^{l_3} S_{j_i}}$ of it discussed and chosen in the space ${\bigcap}_{i=1}^{l_3} S_{j_i}$, we have the relation $U_{{x_{\rm O}}^{\prime},{\bigcap}_{i=1}^{l_3} S_{j_i}} \subset S_{0,\{j_i\}_{i=1}^{l_3}} \bigcap \overline{D}$. 

The set $S_{0,\{j_i\}_{i=1}^{l_3}}$ and $S_{0,\{j_i\}_{i=1}^{l_3}} \bigcap \overline{D}$
 are open sets in ${\bigcap}_{i=1}^{l_3} S_{j_i}$ and ($n-l_3$)-dimensional manifolds with no boundary. They are also submanifolds of the ($n-l_3$)-dimensional manifold ${\bigcap}_{i=1}^{l_3} S_{j_i}$. Furthermore, the manifold ${\bigcap}_{i=1}^{l_3} S_{j_i}$ is a regular submanifold and a closed subset of ${\mathbb{R}}^n$ and has no boundary.

This argument is essential in Theorem \ref{thm:2}, presented later.


Let $i_y \in \mathbb{N}_{l_2}-m_{l_1,l_2}(\{j_i\}_{i=1}^{l_3})$. We investigate partial derivatives of the function defined by the real polynomial ${\prod}_{j \in {m_{l_1,l_2}}^{-1}(i_y)} (f_{j}(x_1,\cdots,x_n))-||y_{i_y}||^2={\prod}_{j \in {m_{l_1,l_2}}^{-1}(i_y)} (f_{j}(x_1,\cdots,x_n))-{\Sigma}_{j^{\prime}=1}^{m_{l_2}(i_y)+1} {y_{i_y,j^{\prime}}}^2$. More precisely, we investigate the partial derivative for each variable $y_{i_1,j^{\prime}}:=y_{i_y,j^{\prime}}$ at $(x_{\rm O},y_{\rm O})$.
We have a real number $2y_{i_{1},{j_0}^{\prime}}=2y_{0,i_{1},{j_0}^{\prime}} \neq 0$ for some variable $y_{i_{1},{j_0}^{\prime}}$ at $(x_{\rm O},y_{\rm O})$ where we abuse the notation in STEP 1-1 for example. Let us abuse the same notation. We also have $0$ as the value of the partial derivative at the point for any variable $y_{i_2,{j}^{\prime}}$ satisfying $i_2 \neq i_1$. This is like an argument in STEP 1-1. This counts $l_2-l_3$ of the rank of the differential of the restriction of the map into ${\mathbb{R}}^{l_2}$ defined canonically from
the $l_2>0$ real polynomials. In other words, we choose the $l_2-l_3$ real polynomials from the $l_2$ real polynomials choosing such $l_2-l_3$ elements $i_y$. Remember also that the restriction of $m_{l_1,l_2}$ to $\{j_i\}_{i=1}^{l_3}$ is assumed to be injective in the condition (\ref{mthm:1.3}). 

We choose an integer $i_x \in m_{l_1,l_2}(\{j_i\}_{i=1}^{l_3})$. We discuss partial derivatives of the function defined canonically from the real polynomial ${\prod}_{j \in {m_{l_1,l_2}}^{-1}(i_x)} (f_{j}(x_1,\cdots,x_n))-||y_{i_x}||^2={\prod}_{j \in {m_{l_1,l_2}}^{-1}(i_x)} (f_{j}(x_1,\cdots,x_n))-{\Sigma}_{j^{\prime}=1}^{m_{l_2}(i_x)+1} {y_{i_x,j^{\prime}}}^2$.
We investigate the partial derivatives for the variables $y_{i_1,{j_0}^{\prime}}$ and $y_{i_2,{j}^{\prime}}$, defined just before. We always have $0$ as the values by our definition. 
More precisely, by our definition, at $(x_{\rm O},y_{\rm O})$, $y_{i_x}$ is the origin. We also see that the value of the partial derivative for any variable $y_{i,j}$ is $0$ at the point $(x_{\rm O},y_{\rm O})$.
We also discuss the partial derivative for each variable $x_{i_3}$. For the integer $i_x$, we have the unique number $j_{i_x} \in \{j_i\}_{i=1}^{l_3}$ satisfying $i_x=m_{l_1,l_2}(j_{i_x})$. This follows from the condition (\ref{mthm:1.3}) that the restriction of $m_{l_1,l_2}$ to $\{j_i\}_{i=1}^{l_3}$ is injective. 
The value of the partial derivative is represented as the product of the following two numbers.
\begin{itemize}
	\item The value of the partial derivative of the function $f_{j_{i_x}}(x_1,\cdots,x_n)$ for the variable $x_{i_3}$ at $x_{\rm O}$. 
	Each given hypersurface $S_j$ is assumed to be non-singular and the zero set of the real polynomial $f_j$. 
	Thus the value of the partial derivative is not $0$ for some variable $x_{i_3}:=x_{i_{3,0}}$ ($1 \leq i_{3,0} \leq n$).
	\item The product of the real numbers obtained as the values of the canonically defined real polynomial functions in the family $\{f_{j}(x_1,\cdots,x_n)\}_{j \in {m_{l_1,l_2}}^{-1}(i_x)-\{j_{i_x}\}}$ at $x_{\rm O}$. By our definitions and assumptions with our construction, we can see that this is not $0$.
	\end{itemize}
In addition, remember the assumption on the transversality on intersections for hypersurfaces $S_j$ or the condition (\ref{mthm:1.2}). 
This counts $l_3$ of the rank of the differential of the restriction of the map into ${\mathbb{R}}^{l_2}$ defined canonically from
the $l_2>0$ real polynomials. 
In other words, we choose the $l_3$ real polynomials from the $l_2$ real polynomials respecting such elements $i_x$ and $j_{i_x}$.
This also counts "$l_3$" of the rank independently from the "$l_2-l_3$" before. This is thanks to the following.

\begin{itemize} 
\item At $(x_{\rm O},y_{\rm O})$, $y_{i_x}$ is the origin.
\item For the $l_3$ polynomials the values of the partial derivatives for the variables $y_{i_1,{j_0}^{\prime}}$, $y_{i_2,{j}^{\prime}}$ and an arbitrary variable $y_{i,j}$ are always $0$.
\end{itemize} 

Integrating these arguments, we can see that at $(x_{\rm O},y_{\rm O})$, the differential of the restriction of the map into ${\mathbb{R}}^{l_2}$ defined canonically from
 the $l_2>0$ real polynomials is of rank $l_2>0$.
We have a result like the one in STEP 1-1. \\
\ \\
STEP 1-3 Additional arguments to show that the set $S_0$ is a smooth compact submanifold of $U_D \times {\mathbb{R}}^{m-n+l_2}$ and ${\mathbb{R}}^{m+l_2} \supset U_D \times {\mathbb{R}}^{m-n+l_2}$ with no boundary: our argument here also shows the relation $S_0 \subset \overline{D} \times {\mathbb{R}}^{m-n+l_2} (\subset U_D \times {\mathbb{R}}^{m-n+l_2})$ for example. \\
\ \\
By the assumption on $U_D$, which is an arbitrary sufficiently small connected open neighborhood of the closure $\overline{D}$ of a nicely given connected open set $D \subset {\mathbb{R}}^n$, we can see that we cannot take a point $x \in U_D-\overline{D}$ as the component of $(x,y) \in S_0$. We discuss this precisely.
First we choose an arbitrary point $x_{U} \in U_D-\overline{D}$. By the rule of choosing our open neighborhood $U_D$ of the compact and connected set $\overline{D}$,
we can discuss $U_D$ and related important subsets as follows.
\begin{itemize}
\item For each point $p$ in $\overline{D}-D \subset {\bigcup}_{j=1}^l S_j$, we can choose a sufficiently small open disk $D_p$ in ${\mathbb{R}}^n$ and we do. We can also regard $p$ as a point as follows.
\begin{itemize}
\item It holds that $p \in {\bigcap}_{i=1}^{i_{p,0}} S_{j_i} \bigcap \overline{D}$ for some increasing sequence $\{j_i\}_{i=1}^{i_{p,0}}$.
\item It also holds that $p \notin {\bigcap}_{i=1}^{i_{p,0}+1} S_{{j^{\prime}}_{i}}$ for any increasing sequence $\{{j^{\prime}}_i\}_{i=1}^{i_{p,0}+1}$ containing the sequence $\{j_i\}_{i=1}^{i_{p,0}}$ as a subsequence. 
\end{itemize}
Remember our definitions and conditions on the compact set $
\overline{D}$ and the zero sets $S_j$
for example. We can also say that at any point of $D_p$, the value of $f_j$ is positive for any $j \notin \{j_i\}_{i=1}^{i_{p,0}}$.
\item 
For the compact and connected subset $\overline{D}$, we can find the open cover $\{D\} \sqcup \{D_p\}_{p \in \overline{D}-D}$ . We can choose finitely many open sets to cover this compact and connected set $\overline{D}$. We choose such sets.
\item (Different from "the original $U_D$",) we can choose $U_D$ as the open set $U_D:={U_{D}}^{\prime}$ as the union of the finitely many open sets chosen before instead. It can and must be chosen as a connected set by our definitions, conditions and construction. We may also choose our new $U_D:={U_D}^{\prime \prime}$ as an arbitrary open connected subset of ${U_D}^{\prime}$ if it is also an open neighborhood of $\overline{D}$.
\end{itemize}

By the condition that the restriction of $m_{l_1,l_2}$ to $\{j_i\}_{j=1}^{l_3}$ is injective, we can choose an arbitrary point $x_{U} \in U_D-\overline{D}$ in such a way that we also have the following.
\begin{itemize}
\item For some sufficiently small open disk $D_{p_0}$ in the family of the chosen finitely many open disks just before, $x_U \in D_{p_0}$.
\item There exists an integer $1 \leq i_{x_{U,0}} \leq l_2$ with the following properties.
\begin{itemize} 
\item It holds that $f_{j}(x_{U})>0$ for each $j \in  {m_{l_1,l_2}}^{-1}(i_{x_{U,0}})$ except one number $j=j_{x_{U,0}} \in {m_{l_1,l_2}}^{-1}(i_{x_{U,0}})$.
\item It also holds that $f_{j_{x_{U,0}}}(x_{U})<0$.
\end{itemize}
\end{itemize}
Here we also see that the value ${\prod}_{j \in {m_{l_1,l_2}}^{-1}(i_{x_{U,0}})} f_{j}(x_{U})$ is smaller than $0$. 

This means that we cannot take a point $x \in U_D-\overline{D}$ satisfying $(x,y) \in S_0$.

By the argument and STEPs 1-1 and 1-2 with implicit function theorem, the set $S_0$ is an $m$-dimensional smooth compact and connected submanifold in $U_D \times {\mathbb{R}}^{m-n+l_2}$ and has no boundary. This is also regarded as a non-singular real algebraic manifold in the real affine space ${\mathbb{R}}^{m+l_2}$. \\
\ \\
STEP 2 Our desired map $f:M \rightarrow {\mathbb{R}}^n$ on $M:=S_0$. \\
We restrict the canonical projection ${\pi}_{m+l_2,n}$ to $M$. Thus we have a smooth real algebraic map $f:M:=S_0 \rightarrow {\mathbb{R}}^n$. We check that $f$ is our desired map.
 
Thanks to our previous arguments, our non-singular real algebraic manifold $M=S_0$ is a connected component of the intersection of finitely many zero sets of real polynomials (in the real affine space).
STEP 1-3 and our definition of $S_0$ also show that the image $f(M)$ is a subset of $U_D$ containing $\overline{D}$ as a subset. Each point of $U_D-\overline{D}$ is also shown to be outside the image. We can see $f(M)=\overline{D}$. 
On the smooth manifolds of the preimages (in Theorem \ref{thm:1}), we can also easily see our desired fact from our definitions, conditions and construction. 

We can see that $f$ is our desired map. 

We can choose $M$ as a connected manifold in the case the function $m_{l_2}$  is always positive, easily: for example in this case preimages are always connected (if they are not empty).
\\
%




This completes the proof.
	
\end{proof}

We can know the following easily by the construction and local structures of the functions and the maps. We can also say that this is also thanks to an argument presented in STEP 1-2 and announced to be essential there. 

\begin{Thm}
\label{thm:2}
In Main Theorem \ref{mthm:1} and Theorem \ref{thm:1}, we can have our map $f:M \rightarrow {\mathbb{R}}^n$ enjoying the following property.
\begin{enumerate}
\item The image $f(S(f))$ of the singular set of $f$ satisfies the relation $f(S(f))=\overline{D}-D$. 
\item Choose any point $p \in {\bigcap}_{i=1}^{l_3} S_{j_i}$ such that for any increasing sequence $\{{j^{\prime}}_i\}_{i=1}^{l_3+1}$ containing the sequence $\{j_i\}_{i=1}^{l_3}$, the condition $p \notin {\bigcap}_{i=1}^{l_3+1} S_{{j^{\prime}}_i}$ is satisfied. Each point $q \in f^{-1}(p)$ satisfies the relation ${df}_q(T_qM)=T_p({\bigcap}_{i=1}^{l_3} S_{j_i})$.
\end{enumerate}
\end{Thm}

Note again that \cite{kitazawa3} presents a specific case for Main Theorem \ref{mthm:1}. These hypersurfaces $S_j$ do not intersect and $l_2=1$ there.

\subsection{A remark on structures of obtained maps.}

We present exposition on global structures of obtained maps. We do not need to understand this rigorously in the present paper.

In the case where the hypersurfaces do not intersect, the obtained map is a so-called {\it special generic} map. There exist such a nice map $f_0:M \rightarrow {\mathbb{R}}^n$ and a diffeomorphism $\Phi:M \rightarrow M$ enjoying the relation $f_0 \circ \Phi=f$.
It is first presented as a remark which is not essential in the paper.
Discussing this rigorously and carefully, the author has found that this may not be trivial. \cite{kitazawa8} proves this rigorously first, after earlier versions of the present paper have appeared.

In the case where the hypersurfaces may intersect, the map is (, at least topologically,) locally regarded as a so-called {\it moment map}.

Here we adopt methods for some presentations in \cite{kitazawa7} for example.

{\it Special generic} maps are, in short, higher dimensional versions and generalizations of the canonical projections of the unit spheres and Morse functions on spheres with exactly two singular points.
For special generic maps, see \cite{saeki1} and see also recent preprints \cite{kitazawa5, kitazawa10} of the author. Of course, arguments and results in \cite{kitazawa5, kitazawa10} and \cite{kitazawa6}, presented later, are independent of our study and we do not need to understand them precisely. We only give related short remark.

First, see \cite{buchstaberpanov, delzant} for moment maps on so-called {\it symplectic toric} manifolds. \cite{kitazawa6} is a preprint introducing a certain class of smooth maps generalizing the class of special generic maps first. There the class of {\it simply generalized special generic} maps has been introduced. It also investigates their topological properties, especially the cohomology rings of the manifolds. This respects and extends the results of \cite{kitazawa10}. These maps are locally moment maps. We can have maps locally, at least topologically, special generic maps or simply generalized special generic maps being not special generic in cases where the hypersurfaces $S_j$ do not intersect. Our main result of \cite{kitazawa3, kitazawa9} is for an explicit and simplest case and for special generic cases.

For related singularity theory of differentiable maps, see \cite{golubitskyguillemin} for example. As presented in the assumption of Main Theorem \ref{mthm:1}, the notion "transversality" is also from this theory, for example.

For Morse-Bott functions, see \cite{bott} for example. We present its definition in (the second subsection of) the third section as a kind of short review. Some Morse-Bott functions are generalized to the presented maps as higher dimensional versions. 

For here, see also \cite{kohnpieneranestadrydellshapirosinnsoreatelen}. This is on explicit classifications of regions surrounded by non-singular real algebraic hypersurfaces intersecting satisfying the condition on the "transversality" like our case. Related to this, \cite{bodinpopescupampusorea, sorea1, sorea2} are mainly on regions surrounded by non-singular real algebraic hypersurfaces with no intersections. This appears explicitly in Example 1 and FIGURE 1 of \cite{kitazawa3}.

\section{Applications of Main Theorem \ref{mthm:1} and Theorems \ref{thm:1} and \ref{thm:2}: functions represented as the compositions of maps obtained through (arguments in our proof of) them with the canonical projections.}
\subsection{Graphs and Reeb graphs.} 
({\it Reeb}) {\it graphs} are our fundamental tools. 

A {\it graph} can be defined as a $1$-dimensional CW complex with the {\it vertex set} and the {\it edge set}. They are the set of all $0$-dimensional cells and the set of all $1$-dimensional cells, respectively. A {\it vertex} is an element of the vertex set. An {\it edge} is an element of the edge set. 
The closure of an edge homeomorphic to a circle is a {\it loop}. We do not discuss graphs having loops. In other words, a graph is always regarded as a $1$-dimensional simplicial complex and a polyhedron. Furthermore, it is {\it finite}. In a word, its vertex set and edge set are finite. On the other hand, a graph may be a so-called multi-graph. This means that a graph may have an edge connecting two distinct vertices. An {\it isomorphism} from a graph $K_1$ onto another graph $K_2$ means a piecewise smooth homeomorphism mapping the edge set and the vertex set of $K_1$ onto those of $K_2$. This defines a natural equivalence relation on the family of all graphs here. Two graphs $K_1$ and $K_2$ are {\it isomorphic} if an isomorphism from $K_1 $ onto $K_2$ exists. 

For a smooth function $c:X \rightarrow \mathbb{R}$, we can define an equivalence relation by the rule that two points $x_1$ and $x_2$ in $X$ are equivalent if and only if they are in a same connected component of a preimage $c^{-1}(y)$ ($y \in \mathbb{R}$). Let $W_c:=X/ {\sim}_c$ denote the quotient space of $X$ under this relation $\sim c$. Let $q_c:X \rightarrow W_c$ denote the quotient map.

\begin{Def}
If the Reeb space $W_c$ has the structure of a graph by the following rule, then $W_c$ is the {\it Reeb graph} of $c$. A point $p$ is a vertex if and only if ${q_c}^{-1}(p)$ contains some singular points of $c$.
\end{Def}
We introduce Theorem 3.1 of \cite{saeki2}.
For a smooth function on a compact manifold having finitely many singular values of it, the quotient space $W_c$ is homeomorphic to a graph. If the manifold of the domain is closed, then we can define the Reeb graph $W_c$ of $c$. {\rm Morse}({\rm -Bott}) functions and functions of some considerably wide classes satisfy such conditions.

The Reeb graph of a smooth function has been already defined in \cite{reeb} for example. The Reeb graphs of nice smooth functions have been important tools and objects in the theory on singularities and applications to geometry. These graphs have some important information on the manifolds compactly.

\subsection{Functions represented as the compositions of maps obtained through Main Theorem \ref{mthm:1} and Theorems \ref{thm:1} and \ref{thm:2} with the canonical projections to the $1$-dimensional real affine space.}

We discuss functions represented as the compositions of maps obtained through (arguments in our proof of) Main Theorem \ref{mthm:1} and Theorems \ref{thm:1} and \ref{thm:2}.
We mainly investigate cases where the hypersurfaces $S_j$ are circles in the plane ${\mathbb{R}}^2$.
 
Hereafter, we need elementary knowledge on Morse(-Bott) functions. We omit the definition of a Morse function. A {\it Morse-Bott} function is a smooth function which is represented as the composition of a submersion with a Morse function around each singular point (for suitable local coordinates).


Hereafter, a {\it circle} means a circle in ${\mathbb{R}}^2$ centered at a point $p_0:=(p_{0,1},p_{0,2}) \in {\mathbb{R}}^2$ and of a radius $r_0>0$ unless otherwise stated. This is represented as the zero set of the real polynomial ${(x_1-p_{0,1})}^2+{(x_2-p_{0,2})}^2-{r_0}^2$ ($x=(x_1,x_2) \in {\mathbb{R}}^2$). Of course this is non-singular. We call the point $(p_{0,1} \pm r,p_{0,2})$ ($(p_{0,1},p_{0,2} \pm r)$) a {\it vertical} (resp. {\it horizontal}) {\it pole} of the circle. 

\begin{Prop}
\label{prop:1}
In Main Theorem \ref{mthm:1} and Theorems \ref{thm:1} and \ref{thm:2}, let $n=2$ and suppose that the family $\{S_j\}_{j=1}^{l_1} \subset {\mathbb{R}}^2$ is a family of mutually disjoint circles and that the relation $\overline{D}-D={\sqcup}_{j=1}^{l_1} S_j$ holds. We do our construction of $f:M \rightarrow {\mathbb{R}}^2$ as presented in our proof. The function $f_0:={\pi}_{2,1} \circ f:M \rightarrow \mathbb{R}$ enjoys the following.
\begin{enumerate}
\item Let $S_j$ be a circle in ${\mathbb{R}}^2$ centered at a point $p_j:=(p_{j,1},p_{j,2}) \in {\mathbb{R}}^2$ and of a radius $r_j>0$.
The singular set $S(f_0) \subset M$ of the function $f_0$ is represented as $f^{-1}({\sqcup}_{j=1}^{l_1} \{(p_{j,1}-r_j,p_{j,2}),(p_{j,1}+r_j,p_{j,2})\})$.
\item The function $f_0$ is a Morse-Bott function. In the case $l_2=1$, we can have our function $f_0$ as a Morse function.
\end{enumerate}

\end{Prop}

\begin{proof}
Most of our idea for our proof is presented in our preprint \cite{kitazawa9} and the proof of "Main Theorem \ref{mthm:2} there".
We can also argue from "Discussion 14 of \cite{kollar}" as in \cite{kitazawa9}. We can discuss \cite{kollar} independently. We do not need to understand the arguments of the preprint \cite{kitazawa9}.

By the presented arguments, the singular set $S(f_0)$ is represented as a subset of $f^{-1}({\sqcup}_{j=1}^{l_1} \{(p_{j,1}-r_j,p_{j,2}),(p_{j,1}+r_j,p_{j,2})\})$. We can also understand this from our construction here.

Let $1 \leq j_0 \leq l_1$ be an integer. Let us abuse the notation $y=(y_1,\cdots y_{l_2})$, used in our previous proof of Main Theorem \ref{mthm:1} with Theorems \ref{thm:1} and \ref{thm:2}. We also abuse our notation there other than this in our present proof. 
We choose a sufficiently small positive number $a_{j,0}>0$ and an arbitrary number $t_0 \in [0,1]$ of the interval $[0,1]$.
We check the zero set of the polynomial ${\prod}_{j \in {m_{l_1,l_2}}^{-1}(m_{l_1,l_2}(j_0))} (f_{j}(x_1,x_2))-a_{j,0} t_0$. The relation $a_{j,0} t_0={||y_{m_{l_1,l_2}(j_0)}||}^2$ holds at a point $(x,y) \in \overline{D} \times {\mathbb{R}}^{m-n+l_2} \subset {\mathbb{R}}^{m+l_2}$ ($x=(x_1,x_2)$) in this zero set. We apply some of our arguments in the proof of "Main Theorems 1 and 2 of \cite{kitazawa8}" or the arguments from \cite{kollar}. Thus, around the zero set of the polynomial $f_{j}=f_{j_0}$, our map $f:M \rightarrow {\mathbb{R}}^2$ is represented as the product map of a natural Morse function on the disk $D^{m_2 \circ m_{l_1,l_2}(j_0)+1}$ and the identity map on the product of the zero set of the polynomial $f_{j_0}$ and a sphere diffeomorphic to the product ${\prod}_{j \in {\mathbb{N}_{l_2}}-\{m_{l_1,l_2}(j_0)\}} S^{m_2(j)}$ in the case $l_2>1$ and the product map of a natural Morse function on the disk $D^{m_2 \circ m_{l_1,l_2}(j_0)+1}$ and the identity map on the zero set of the polynomial $f_{j_0}$ in the case $l_2=1$. 

We discuss these functions and maps more precisely.


The identity map on ${\prod}_{j \in {\mathbb{N}_{l_2}}-\{m_{l_1,l_2}(j_0)\}} S^{m_2(j)}$ comes from a natural map mapping each point $(x,y) \in M \bigcap (\overline{D} \times {\mathbb{R}}^{m-n+l_2})=M \bigcap (\overline{D} \times {\mathbb{R}}^{m_2 \circ m_{l_1,l_2}(j_0)+1} \times {\mathbb{R}}^{m-n+l_2-m_2 \circ m_{l_1,l_2}(j_0)-1} \subset {\mathbb{R}}^{m+l_2})$ which is also located around the zero set of $f_{j_0}$ to an element of $\overline{D} \times {\mathbb{R}}^{m_2 \circ m_{l_1,l_2}(j_0)+1} \times {\prod}_{j \in {\mathbb{N}_{l_2}}-\{m_{l_1,l_2}(j_0)\}} S^{m_2(j)}$.

We explain this natural map explicitly.
We remember the notation $y=(y_1,\cdots y_{l_2})$. For $i \in {\mathbb{N}_{l_2}}-\{m_{l_1,l_2}(j_0)\}$, the component $y_i$ is mapped onto the component $\frac{1}{\sqrt{{\prod}_{j \in {m_{l_1,l_2}}^{-1}(i)} (f_{j}(x_1,x_2))}} y_i \in S^{m_2(i)}$ at $(x,y)$. The component $x$ is mapped to $x$ itself and the component $y_{m_{l_1,l_2}(j_0)}$ is mapped to $y_{m_{l_1,l_2}(j_0)}$ itself. Around the zero set of the polynomial $f_{j_0}$ and its preimage by the map $f$, this map also gives a diffeomorphism.

Remember that restrictions of the projection ${\pi}_{2,1}$ to circles are Morse functions. We can show this as a kind of elementary exercises on Morse functions.

These arguments complete the proof. 

\end{proof}
We present  Main Theorem \ref{mthm:2}. This gives some explicit families of real algebraic functions enjoying the following properties.
\begin{itemize}
\item These functions are obtained as compositions of real algebraic maps into ${\mathbb{R}}^2$ in Main Theorem \ref{mthm:1} and Theorems \ref{thm:1} and \ref{thm:2} with ${\pi}_{2,1}$.
\item These functions are Morse-Bott.
\item The images of real algebraic maps into ${\mathbb{R}}^2$ before are bounded and connected regions in ${\mathbb{R}}^2$ surrounded by circles intersecting satisfying transversality.
\item The Reeb graphs of the functions collapse to (given) graphs. For example, ones isomorphic to the Reeb graphs of Morse-Bott functions in Proposition \ref{prop:1}.
\item The manifolds $M$ in these theorems are the zero sets of some real polynomial maps.
\end{itemize}
We first define the class of maps constructed in (proving) Main Theorem \ref{mthm:1} and Theorems \ref{thm:1} and \ref{thm:2}.
\begin{Def}
In Main Theorem \ref{mthm:1} and Theorems \ref{thm:1} and \ref{thm:2}, 
the connected component $S_0:=\{(x,y_1,\cdots,y_{l_2}) \in U_D \times {\prod}_{i=1}^{l_2} {\mathbb{R}}^{m_{l_2}(i)+1} \subset {\mathbb{R}}^{m+l_2} \mid {\prod}_{j \in {m_{l_1,l_2}}^{-1}(i)} (f_{j}(x_1,\cdots,x_n))-{||y_i||}^2=0, i \in {\mathbb{N}}_{l_2}\} \subset {\mathbb{R}}^{m+l_2}$ of the zero set of a real algebraic map is defined. We have also defined the $m$-dimensional non-singular real algebraic manifold $M:=S_0$.  The map $f$ is defined as the composition of the canonical embedding into ${\mathbb{R}}^{m+l_2}$ with ${\pi}_{m+l_2,n}$.
We can canonically define $M$ and $f$ from the data $(D, \{f_{j}\}_{j=1}^{l_1},m_{l_1,l_2},m_{l_2})$. Let us use the notation $M:=M_{(D,\{f_{j}\}_{j=1}^{l_1},m_{l_1,l_2},m_{l_2})}$ and $f:=f_{(D,\{f_{j}\}_{j=1}^{l_1},m_{l_1,l_2},m_{l_2})}$.
We also call the map $f$ the {\it moment-like} map reconstructed from $(D, \{f_{j}\}_{j=1}^{l_1},m_{l_1,l_2},m_{l_2})$. 
\end{Def}
Let the composition of such a map with ${\pi}_{n,1}$, which is a smooth function, denoted by $f_{0,(D,\{f_{j}\}_{j=1}^{l_1},m_{l_1,l_2},m_{l_2})}$. 
Hereafter, we respect the case each $S_j$ is a circle unless otherwise stated. We use the notation for "the set $S_j$" instead of that for "the polynomial $f_j$" in such a case. This contains no problems. More rigorously, here the circle $S_j$ is for a polynomial ${(x_1-p_{0,1})}^2+{(x_2-p_{0,2})}^2-{r_0}^2$ or ${r_0}^2-{(x_1-p_{0,1})}^2-{(x_2-p_{0,2})}^2$ ($x=(x_1,x_2) \in {\mathbb{R}}^2$). Each of these two types is chosen in such a way that this is compatible with our open set $D$ and Theorem \ref{mthm:1} (\ref{mthm:1.1}).
\begin{MainThm}
\label{mthm:2}
Given a situation as in Proposition \ref{prop:1}, in Main Theorem \ref{mthm:1} and Theorems \ref{thm:1} and \ref{thm:2} satisfying the following. Let $n=2$, $l_2:=l_{0,2}=1$ and suppose that the family $\{S_j:=S_{0,j}\}_{j=1}^{l_{0,1}} \subset {\mathbb{R}}^2$ {\rm (}$l_1:=l_{0,1}${\rm )} is a family of mutually disjoint circles and that the relation $\overline{D_0}-D_0={\sqcup}_{j=1}^{l_{0,1}} S_j$ {\rm (}$D:=D_0${\rm )} holds. 
We define the moment-like map $f=f_{(D, \{S_{j}\}_{j=1}^{l_1},m_{l_1,l_2},m_{l_2})}:=f_{(D_0, \{S_{0,j}\}_{j=1}^{l_{0,1}},m_{l_{0,1},l_{0,2}},m_{l_{0,2}})}:M_{(D, \{S_{j}\}_{j=1}^{l_1},m_{l_1,l_2},m_{l_2})}:=M_{(D_0, \{S_{0,j}\}_{j=1}^{l_{0,1}},m_{l_{0,1},l_{0,2}},m_{l_{0,2}})} \rightarrow {\mathbb{R}}^2$. For the Reeb graph $W_{f_{0,(D_0, \{S_{0,j}\}_{j=1}^{l_{0,1}},m_{l_{0,1},l_{0,2}},m_{l_{0,2}})}}$ of the function $f_{0,(D_0, \{S_{0,j}\}_{j=1}^{l_{0,1}},m_{l_{0,1},l_{0,2}},m_{l_{0,2}})}$, let $n_e$ denote the number of the edges of $W_{f_{0,(D_0, \{S_{0,j}\}_{j=1}^{l_{0,1}},m_{l_{0,1},l_{0,2}},m_{l_{0,2}})}}$ and $\{e_j\}_{j=1}^{n_e}$ the family of all edges of the graph.
Here let  $m_{l_{0,2}}$ denote not only the function but also the value $m_{l_{0,2}}(1)${\rm :} note that the function $m_{l_{0,2}}$ is constant here. 

Then, for each map $l_{\{e_j\}_{j=1}^{n_e}}:{\mathbb{N}}_{n_e} \rightarrow {\mathbb{N}} \sqcup \{0\}$ and an arbitrary integer $m_{l_{\{e_j\}_{j=1}^{n_e}}}>m_{l_{0,2}}+2$, there exist a non-singular real algebraic manifold $M_{l_{\{e_j\}_{j=1}^{n_e}}}$ which is also a closed and connected manifold, which is the zero set of a real polynomial map, and whose dimension is $m_{l_{\{e_j\}_{j=1}^{n_e}}}$, and a smooth real algebraic map $f_{l_{\{e_j\}_{j=1}^{n_e}}}:M_{l_{\{e_j\}_{j=1}^{n_e}}} \rightarrow {\mathbb{R}}^2$ with the following properties. 
\begin{enumerate}
\item \label{mthm:2.1} Each map $f:=f_{l_{\{e_j\}_{j=1}^{n_e}}}$ is  the moment-like map reconstructed from some data $(D,\{S_j\}_{j=1}^{l_1},m_{l_1,l_2},m_{l_2}):=(D_{l_{\{e_j\}_{j=1}^{n_e}}},\{S_{l_{\{e_j\}_{j=1}^{n_e}},j}\}_{j=1}^{l_1},m_{l_{\{e_j\}_{j=1}^{n_e}},l_1,l_2},m_{l_{\{e_j\}_{j=1}^{n_e}},l_2})$ by choosing each $S_{l_{\{e_j\}_{j=1}^{n_e}},j}$. The following conditions are satisfied.
\begin{enumerate}
\item The non-empty open set $D:=D_{l_{\{e_j\}_{j=1}^{n_e}}}$ is a connected proper subset of $D_0$.
The intersection $(\overline{D}-D) \bigcap S_j=(\overline{D_{l_{\{e_j\}_{j=1}^{n_e}}}}-D_{l_{\{e_j\}_{j=1}^{n_e}}}) \bigcap S_j$ is non-empty for each $1 \leq j \leq l_{1,1}$.
\item We define the two integers $l_1:=l_{0,1}+{\Sigma}_{j^{\prime}=1}^{n_e} l_{\{e_j\}_{j=1}^{n_e}}(j^{\prime})$ and $l_2:=l_{0,2}+1=1+1=2$.
\item Each $S_j$ is a circle.  For $1 \leq j \leq l_{0,1}$, $S_j:=S_{j,0}$.
\item Distinct circles in the family $\{S_{l_{0,1}+j}\}_{j=1}^{l_1-l_{0,1}}$ are disjoint. For each circle $S_{l_{0,1}+j_2}$ {\rm (}$1 \leq j_2 \leq l_1-l_{0,1}${\rm )} in this family, there exists a unique circle $S_{j_1,0}$ in the family $\{S_{j,0}\}_{j=1}^{l_{0,1}}$ such that the intersection $S_{j_1,0} \bigcap S_{l_{0,1}+j_2}$ is not empty. Furthermore, the intersection is a discrete two-point set. 
\end{enumerate}
Hereafter, let $f_{0,l_{\{e_j\}_{j=1}^{n_e}}}:={\pi}_{2,1} \circ f_{l_{\{e_j\}_{j=1}^{n_e}}}$.

\item \label{mthm:2.2} Each function $f_{0,l_{\{e_j\}_{j=1}^{n_e}}}$ is a Morse-Bott function.
\item \label{mthm:2.3} Each Reeb graph $W_{f_{0,l_{\{e_j\}_{j=1}^{n_e}}}}$ collapses to the original Reeb graph $W_{f_{0,(D_0, \{S_{0,j}\}_{j=1}^{l_{0,1}},m_{l_{0,1},l_{0,2}},m_{l_{0,2}})}}$. 
\item \label{mthm:2.4}
For any two distinct cases such that the values of the maps $l_{\{e_j\}_{j=1}^{n_e}}:=l_{\{e_j\}_{j=1}^{n_e},1}$ and $l_{\{e_j\}_{j=1}^{n_e}}:=l_{\{e_j\}_{j=1}^{n_e},2}$ are same except at exactly one number in ${\mathbb{N}}_{n_e}$, the Reeb graphs $W_{f_{0,l_{\{e_j\}_{j=1}^{n_e},1}}}$ and $W_{f_{0,l_{\{e_j\}_{j=1}^{n_e},2}}}$ are not isomorphic.

\end{enumerate}
\end{MainThm}
\begin{proof}

We choose circles $S_j$ for our proof.
We can choose our circles enjoying the following properties.
\begin{enumerate}\setcounter{enumi}{4}
\item \label{mthm:2.5} For $1 \leq j \leq l_{0,1}$, $S_j:=S_{j,0}$.
\item \label{mthm:2.6} Circles in the family $\{S_{l_{0,1}+j}\}_{j=1}^{l_1-l_{0,1}}$ are centered at points in some circles in the family $\{S_{j,0}\}_{j=1}^{l_{0,1}}$ which are not vertical or horizontal poles of the circles. These circles are sufficiently small and mutually disjoint.
We remember the map $f_{(D_0, \{S_{0,j}\}_{j=1}^{l_{0,1}},m_{l_{0,1},l_{0,2}},m_{l_{0,2}})}$, the function $f_{0,(D_0, \{S_{0,j}\}_{j=1}^{l_{0,1}},m_{l_{0,1},l_{0,2}},m_{l_{0,2}})}$, and the quotient map $q_{f_{0,(D_0, \{S_{0,j}\}_{j=1}^{l_{0,1}},m_{l_{0,1},l_{0,2}},m_{l_{0,2}})}}$. We take the preimage $P_{e_{i}}:={q_{f_{0,(D_0, \{S_{0,j}\}_{j=1}^{l_{0,1}},m_{l_{0,1},l_{0,2}},m_{l_{0,2}})}}}^{-1}(e_{i}) \subset M_{(D_0, \{S_{0,j}\}_{j=1}^{l_{0,1}},m_{l_{0,1},l_{0,2}},m_{l_{0,2}})}$, the image $f_{(D_0, \{S_{0,j}\}_{j=1}^{l_{0,1}},m_{l_{0,1},l_{0,2}},m_{l_{0,2}})}(P_{e_i})$ and the intersection \\ $f_{(D_0, \{S_{0,j}\}_{j=1}^{l_{0,1}},m_{l_{0,1},l_{0,2}},m_{l_{0,2}})}(P_{e_i}) \bigcap {\sqcup}_{j=1}^{l_{0,1}} S_j$. 
Remember also Proposition \ref{prop:1} and that the function $f_{0,(D_0, \{S_{0,j}\}_{j=1}^{l_{0,1}},m_{l_{0,1},l_{0,2}},m_{l_{0,2}})}$ is a Morse function for example.
The intersection $f_{(D_0, \{S_{0,j}\}_{j=1}^{l_{0,1}},m_{l_{0,1},l_{0,2}},m_{l_{0,2}})}(P_{e_i}) \bigcap {\sqcup}_{j=1}^{l_{0,1}} S_j$ is not empty and it is a disjoint union of two curves diffeomorphic to $\mathbb{R}$. We choose one connected curve $L_i$ from these curves.
The circle $S_{l_{0,1}+{\Sigma}_{j^{\prime}=1}^{i-1} l_{\{e_j\}_{j=1}^{n_e}}(j^{\prime})+j^{\prime \prime}}$ is centered at a point in $L_i$ for $1 \leq j^{\prime \prime} \leq l_{\{e_j\}_{j=1}^{n_e}}(i)$. Let ${\mathbb{R}^2}_{S_{l_{0,1}+j}}$ be the unbounded connected component of the complementary set ${\mathbb{R}}^2-S_{l_{0,1}+j}$ for $1 \leq j \leq l_1-l_{0,1}$. We define $D=D_{l_{\{e_j\}_{j=1}^{n_e}}}:=D_0 \bigcap {\bigcap}_{j=1}^{l_1-l_{0,1}} {\mathbb{R}^2}_{S_{l_{0,1}+j}}$.
\item \label{mthm:2.7} Each circle $S_{l_{0,1}+j}$ in $\{S_{l_{0,1}+j}\}_{j=1}^{l_1-l_{0,1}}$ contains exactly one vertical pole $v_{S_{l_{0,1}+j}}$ of the circle $S_{l_{0,1}+j}$ satisfying $v_{S_{l_{0,1}+j}} \in \overline{D}$ and two points $p_{S_{l_{0,1}+j},1}$ and $p_{S_{l_{0,1}+j},2}$ contained in exactly two distinct circles in $\{S_j\}_{j=1}^{l_1}$ and $\overline{D}$. These three points are mutually distinct and the values of the first components are also mutually distinct. 

We explain this by simple figures (FIGUREs \ref{fig:1} and \ref{fig:2}). We need to consider exactly four cases for the local curvature of the circles from $\{S_j\}_{j=1}^{l_{0,1}}$. However, by virtue of the symmetry, it is sufficient to consider one case.   
\begin{figure}
	
	\includegraphics[height=75mm, width=100mm]{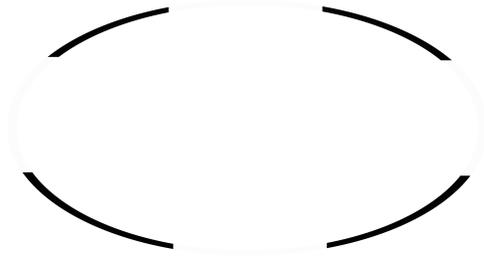}

	\caption{Four cases for the local curvature of a circle.}
	\label{fig:1}
\end{figure}
\begin{figure}
	
	\includegraphics[height=75mm, width=100mm]{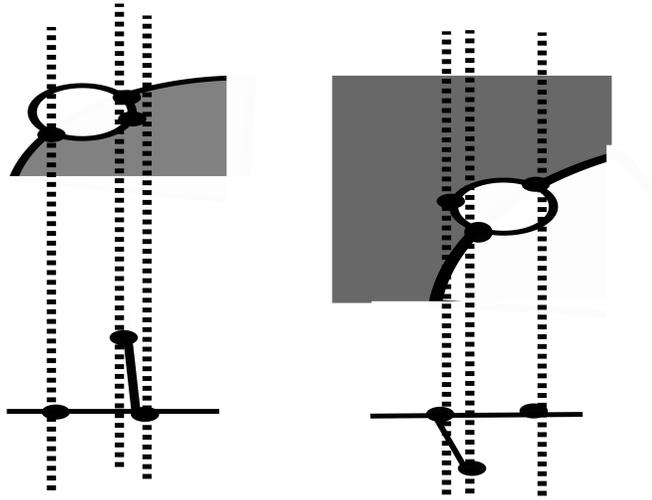}

	\caption{Local intersections of our circles. We also show our Reeb graphs locally. We consider one case.}
	\label{fig:2}
\end{figure}

For exactly two circles containing the two points here, one is from the family $\{S_j\}_{j=1}^{l_{0,1}}$ and the other is from the family $\{S_{l_{0,1}+j}\}_{j=1}^{l_1-l_{0,1}}$.
For distinct circles $S_{l_{0,1}+j_1}$ and $S_{l_{0,1}+j_2}$ from $\{S_{l_{0,1}+j}\}_{j=1}^{l_1-l_{0,1}}$, either of the relation $$\max \{{\pi}_{2,1}(v_{S_{l_{0,1}+j_1}}),{\pi}_{2,1}(p_{S_{l_{0,1}+j_1},1}), {\pi}_{2,1}(p_{S_{l_{0,1}+j_1},2})\}<\min \{{\pi}_{2,1}(v_{S_{l_{0,1}+j_2}}),{\pi}_{2,1}(p_{S_{l_{0,1}+j_2},1}), {\pi}_{2,1}(p_{S_{l_{0,1}+j_2},2})\}$$ or the relation $$\max \{{\pi}_{2,1}(v_{S_{l_{0,1}+j_2}}),{\pi}_{2,1}(p_{S_{l_{0,1}+j_2},1}), {\pi}_{2,1}(p_{S_{l_{0,1}+j_2},2})\}<\min \{{\pi}_{2,1}(v_{S_{l_{0,1}+j_1}}),{\pi}_{2,1}(p_{S_{l_{0,1}+j_1},1}), {\pi}_{2,1}(p_{S_{l_{0,1}+j_1},2})\}$$ holds. Note that the map ${\pi}_{2,1}$ is used to represent the first components.
\item \label{mthm:2.8} Each circle $S_{j,0}$ in $\{S_{j,0}\}_{j=1}^{l_{0,1}}$ contains exactly two vertical poles $v_{S_{j,0},1}$ and $v_{S_{j,0},2}$ of it satisfying $v_{S_{j,0},1}, v_{S_{j,0},2} \in \overline{D}$. 
The values of the first components of these vertical poles are always distinct from those of points discussed in (\ref{mthm:2.7}): the vertical poles $v_{S_{l_{0,1}+j}}$ and points $p_{S_{l_{0,1}+j},1}$ and $p_{S_{l_{0,1}+j},2}$ contained in exactly two distinct circles in $\{S_j\}_{j=1}^{l_1}$.
\item \label{mthm:2.9} We can apply Main Theorem \ref{mthm:1} with Theorems \ref{thm:1} and \ref{thm:2} and our proof to have our new map $f_{l_{\{e_j\}_{j=1}^{n_e}}}:M_{l_{\{e_j\}_{j=1}^{n_e}}} \rightarrow {\mathbb{R}}^2$. More explicitly we set as follows. We put $l_2:=2$.

We set $m_{l_{\{e_j\}_{j=1}^{n_e}},l_1,l_2}$ as a function onto the set $\{1,2\}$ and $m_{l_{\{e_j\}_{j=1}^{n_e}},l_2}$ as a function with the following conditions.
\begin{enumerate}
	\item The restriction $m_{l_{\{e_j\}_{j=1}^{n_e}},l_1,l_2} {\mid}_{{\mathbb{N}}_{l_{0,1}}}$ and the function $m_{l_{0,1},l_{0,2}}$ agree where the sets of the targets are seen as $\{1,2\}$.
	\item The restriction $m_{l_{\{e_j\}_{j=1}^{n_e}},l_1,l_2} {\mid}_{{\mathbb{N}}_{l_{0,1}}}$ is a constant function whose values are always $1$. The restriction $m_{l_{\{e_j\}_{j=1}^{n_e}},l_1,l_2} {\mid}_{{\mathbb{N}}_{l_1}-{\mathbb{N}}_{l_{0,1}}}$ is a constant function whose values are always $2$.
	\item We put $m_{l_{\{e_j\}_{j=1}^{n_e}},l_2}(1)=m_{l_{0,2}}:=m_{l_{0,2}}(1)$.
	We put $m_{l_{\{e_j\}_{j=1}^{n_e}},l_2}(2):=m_{l_{\{e_j\}_{j=1}^{n_e}}}-(m_{l_{0,2}}+2)>0$.
\end{enumerate}

\end{enumerate}

We construct our map $f_{l_{\{e_j\}_{j=1}^{n_e}}}:M_{l_{\{e_j\}_{j=1}^{n_e}}} \rightarrow {\mathbb{R}}^2$ in (\ref{mthm:2.9}). We show that this map is our desired map. 

We can check the property (\ref{mthm:2.1}) here easily from our situation. 

We show the property (\ref{mthm:2.2}). This is on the singularities of the functions. 
We discuss local singularities of the function $f_{0,l_{\{e_j\}_{j=1}^{n_e}}}$. 

For $p \in \overline{D}$, the preimage ${f_{l_{\{e_j\}_{j=1}^{n_e}}}}^{-1}(p)$ contains some singular points of the function $f_{0,l_{\{e_j\}_{j=1}^{n_e}}}$ only if it is a vertical pole $v$ of some circle $S_j$ satisfying $v \in \overline{D}$ or a point contained in exactly two disjoint circles in $\{S_j\}_{j=1}^{l_1}$ and $\overline{D}$. We can see this fact on singularities from the structures of the functions and the maps easily.

We remember our proof of Proposition \ref{prop:1} and related arguments on singularities of the function at vertical poles. These singular points are for singular points of Morse-Bott functions.

We present an argument similar to "our argument for local structures of the map $f$ around the zero set of $f_j$ in the proof of Proposition \ref{prop:1}". This is a higher dimensional version for a non-empty set represented as $S_{j_{0,1}} \bigcap S_{j_{0,2}}$ ($j_{0,1} \neq j_{0,2}$). 

Around each point of the discrete set $S_{j_{0,1}} \bigcap S_{j_{0,2}}$, our map $f:M \rightarrow {\mathbb{R}}^2$ is represented as the product map of a natural Morse function on the disk $D^{m_2 \circ m_{l_1,l_2}(j_{0,1})+1}$, a natural Morse function on the disk $D^{m_2 \circ m_{l_1,l_2}(j_{0,2})+1}$ and the identity map on a sphere diffeomorphic to the product ${\prod}_{j \in {\mathbb{N}_{l_2}}-\{m_{l_1,l_2}(j_{0,1}),m_{l_1,l_2}(j_{0,2})\}} S^{m_2(j)}$ in the case $l_2>2$ and the product map of a natural Morse function on the disk $D^{m_2 \circ m_{l_1,l_2}(j_{0,1})+1}$ and a natural Morse function on the disk $D^{m_2 \circ m_{l_1,l_2}(j_{0,2})+1}$ in the case $l_2=2$. Remember that here $l_2=2$. We have discussed a general case for the singularities.

The identity map on the set ${\prod}_{j \in {\mathbb{N}_{l_2}}-\{m_{l_1,l_2}(j_{0,1}),m_{l_1,l_2}(j_{0,2})\}} S^{m_2(j)}$ comes from a natural map mapping $(x,y) \in M \subset \overline{D} \times {\mathbb{R}}^{m-n+l_2}=\overline{D} \times {\mathbb{R}}^{m_2 \circ m_{l_1,l_2}(j_{0,1})+1} \times {\mathbb{R}}^{m_2 \circ m_{l_1,l_2}(j_{0,2})+1} \times {\mathbb{R}}^{m-n+l_2-m_2 \circ m_{l_1,l_2}(j_{0,1})-m_2 \circ m_{l_1,l_2}(j_{0,2})-2} \subset {\mathbb{R}}^{m+l_2}$ to an element of $\overline{D} \times {\mathbb{R}}^{m_2 \circ m_{l_1,l_2}(j_{0,1})+1} \times {\mathbb{R}}^{m_2 \circ m_{l_1,l_2}(j_{0,2})+1} \times {\prod}_{j \in {\mathbb{N}_{l_2}}-\{m_{l_1,l_2}(j_{0,1}), m_{l_1,l_2}(j_{0,2})\}} S^{m_2(j)}$
around the discrete set.

We present the natural map explicitly.
We remember the notation $y=(y_1,\cdots y_{l_2})$. For $i \in {\mathbb{N}_{l_2}}-\{m_{l_1,l_2}(j_{0,1}),m_{l_1,l_2}(j_{0,2})\}$, the component $y_i$ is mapped onto the component $\frac{1}{\sqrt{{\prod}_{j \in {m_{l_1,l_2}}^{-1}(i)} (f_{j}(x_1,x_2))}} y_i \in S^{m_2(i)}$ at $(x,y)$. The component $x$ is mapped to $x$ itself and the component $y_{m_{l_1,l_2}(j_{0,a})}$ is mapped to $y_{m_{l_1,l_2}(j_{0,a})}$ itself for $a=1,2$. This natural map also gives a diffeomorphism around the discrete set $S_{j_{0,1}} \bigcap S_{j_{0,2}}$ and its preimage by the map $f$.

We respect these arguments and local coordinates of the functions and the maps. We also remember the properties, mainly, (\ref{mthm:2.6}, \ref{mthm:2.7}, \ref{mthm:2.8}, \ref{mthm:2.9}). Here, we also respect these properties, implying that the location of our circles $S_j$ is sufficiently general.
This completes the proof of the property (\ref{mthm:2.2}).

We show the properties (\ref{mthm:2.3}, \ref{mthm:2.4}). We mainly remember the properties (\ref{mthm:2.6}, \ref{mthm:2.7}, \ref{mthm:2.8}, \ref{mthm:2.9}) again.
We investigate the preimage ${f_{0,l_{\{e_j\}_{j=1}^{n_e}}}}^{-1}(t)$ containing no singular points of the function $f_{0,l_{\{e_j\}_{j=1}^{n_e}}}$. We investigate the preimage ${{\pi}_{2,1}}^{-1}(t)$, which is the straight line $L_t \subset {\mathbb{R}}^2$ represented as the zero set of the real polynomial $x_1-t$ ($(x_1,x_2) \in {\mathbb{R}}^2$). We investigate each connected component $L_{e,t}$ of the intersection of the straight line $L_t$ and the image $f_{l_{\{e_j\}_{j=1}^{n_e}}}(M_{l_{\{e_j\}_{j=1}^{n_e}}})$. The interior ${\rm Int}\ L_{e,t}$ is in the interior ${\rm Int}\ f_{l_{\{e_j\}_{j=1}^{n_e}}}(M_{l_{\{e_j\}_{j=1}^{n_e}}})$ of the image and the boundary consists of exactly two points $p_{e,1}$ and $p_{e,2}$ contained in some circles $S_j$ which are mutually distinct. On this, either of the following holds.

\begin{itemize}
\item These points $p_{e,1}$ and $p_{e,2}$ are in distinct circles represented as $S_{0,j}=S_j$ ($1 \leq j \leq l_{0,1}$). The preimage ${f_{l_{\{e_j\}_{j=1}^{n_e}}}}^{-1}(L_{e,t})$ is diffeomorphic to $S^{m_{l_{\{e_j\}_{j=1}^{n_e}},l_2}(1)+1} \times S^{m_{l_{\{e_j\}_{j=1}^{n_e}},l_2}(2)}$. This is also regarded as a manifold diffeomorphic to one obtained by gluing two copies of $D^{m_{l_{\{e_j\}_{j=1}^{n_e}},l_2}(1)+1} \times S^{m_{l_{\{e_j\}_{j=1}^{n_e}},l_2}(2)}$ along the boundaries  in a canonical way. In other words, we take a double of the copy here.
\item One of the points is in a circle represented as $S_{0,j}=S_j$ ($1 \leq j \leq l_{0,1}$) and the other point is in a circle represented as $S_{j}$ ($j>l_{0,1}$). The preimage ${f_{l_{\{e_j\}_{j=1}^{n_e}}}}^{-1}(L_{e,t})$ is diffeomorphic to $S^{m_{l_{\{e_j\}_{j=1}^{n_e}}}-1}$. This is also regarded as a manifold diffeomorphic to one obtained by gluing $D^{m_{l_{\{e_j\}_{j=1}^{n_e}},l_2}(1)+1} \times \partial D^{m_{l_{\{e_j\}_{j=1}^{n_e}},l_2}(2)+1}$ and $\partial (D^{m_{l_{\{e_j\}_{j=1}^{n_e}},l_2}(1)+1}) \times D^{m_{l_{\{e_j\}_{j=1}^{n_e}},l_2}(2)+1}$ along the boundaries in a canonical way.
\end{itemize}
 
 From our situation, we cannot observe the case that these points $p_{e,1}$ and $p_{e,2}$ are in distinct circles represented as $S_{l_{0,1}+j}$ ($1 \leq j \leq l_1-l_{0,1}$).
 
These preimages are also connected components of the preimage ${f_{0,l_{\{e_j\}_{j=1}^{n_e}}}}^{-1}(t)$. See also Remark \ref{rem:1} for a kind of related counterexamples.

We can easily check the property (\ref{mthm:2.3}).

The property (\ref{mthm:2.4}) follows from the difference of numbers of vertices. 

By our construction and arguments, our $M_{l_{\{e_j\}_{j=1}^{n_e}}}$ is connected. Finally, we also show that our $M_{l_{\{e_j\}_{j=1}^{n_e}}}$ is not only a connected component of the zero set of a real polynomial map, but also the zero set. This follows from the fact that the set "$U_D$ is chosen as ${\mathbb{R}}^n$ in Main Theorem \ref{mthm:1}" and the construction in its proof, for example. In Main Theorem \ref{mthm:3}, presented later, we can argue this in the same way.

This completes the proof.
\end{proof}

FIGUREs \ref{fig:3} and \ref{fig:4} show an explicit case for Main Theorem \ref{mthm:2} with $n=2$ and $(l_{0,1},l_{0,2},l_1,l_2,n_e)=(3,1,5,2,5)$.
\begin{figure}
	
	\includegraphics[height=75mm, width=100mm]{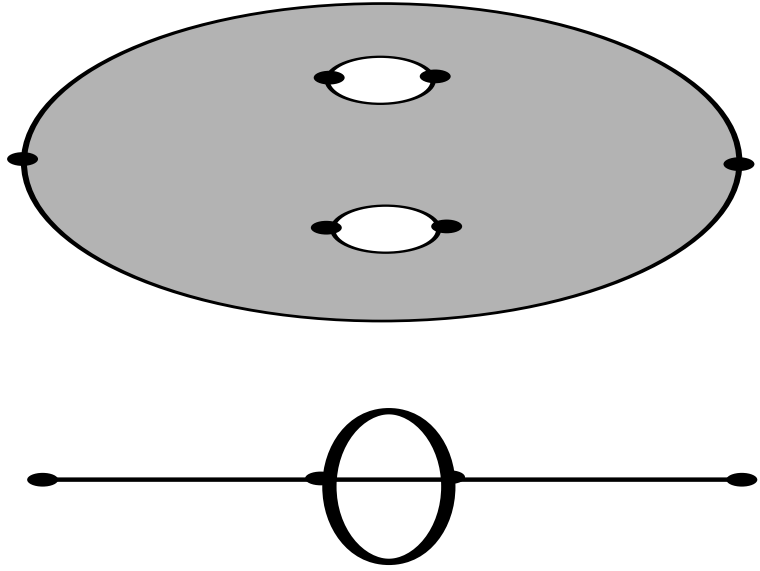}

	\caption{An example for Main Theorem \ref{mthm:2}. The image of a moment-like map $f_{(D_0,\{S_{0,j}\}_{j=1}^{l_{0,1}},m_{l_{0,1},l_{0,2}},m_{l_{0,2}})}$ into ${\mathbb{R}}^2$ for a region $D_0$ surrounded by circles and colored in gray with $n=2$ and $(l_{0,1}.l_{0,2})=(3,1)$. The Reeb graph of the function $f_{0,(D_0,\{S_{0,j}\}_{j=1}^{l_{0,1}},m_{l_{0,1},l_{0,2}},m_{l_{0,2}})}:={\pi}_{2,1} \circ f_{(D_0,\{S_j\}_{j=1}^{l_{0,1}},m_{l_{0,1},l_{0,2}},m_{l_{0,2}})}$. Black dots are for singular points of the function.}
	\label{fig:3}
\end{figure}
\begin{figure}
	
	\includegraphics[height=75mm, width=100mm]{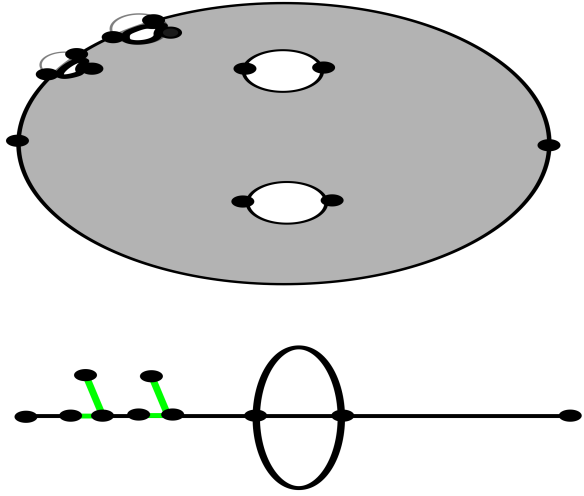}

	\caption{An example for Main Theorem \ref{mthm:2} following FIGURE \ref{fig:3}. 
The image of a moment-like map $f_{(D_{l_{\{e_j\}_{j=1}^{n_e}}},\{S_{l_{\{e_j\}_{j=1}^{n_e}},j}\}_{j=1}^{l_1},m_{l_{\{e_j\}_{j=1}^{n_e}},l_1,l_2},m_{l_{\{e_j\}_{j=1}^{n_e}},l_2})}$ into ${\mathbb{R}}^2$ for a region $D_{l_{\{e_j\}_{j=1}^{n_e}}}$ surrounded by circles and colored in gray: here FIGURE \ref{fig:1} is respected with $(l_{0,1},l_{0,2})=(3,1)$, $l_1=5$, $l_2=2$, $n_e=5$, and $l_{\{e_j\}_{j=1}^{n_e}}$ is a function whose value is $2$ at an edge and whose values are $0$ at the remaining edges. The Reeb graph of the function $f_{0,(D_{l_{\{e_j\}_{j=1}^{n_e}}},\{S_{l_{\{e_j\}_{j=1}^{n_e}},j}\}_{j=1}^{l_1},m_{l_{\{e_j\}_{j=1}^{n_e}},l_1,l_2},m_{l_{\{e_j\}_{j=1}^{n_e}},l_2})}:={\pi}_{2,1} \circ f_{(D_{l_{\{e_j\}_{j=1}^{n_e}}},\{S_{l_{\{e_j\}_{j=1}^{n_e}},j}\}_{j=1}^{l_1},m_{l_{\{e_j\}_{j=1}^{n_e}},l_1,l_2},m_{l_{\{e_j\}_{j=1}^{n_e}},l_2})}$. Black dots are for singular points of the function. For each point in the interior of each black (green) edge, the preimage is diffeomorphic to $S^{m_{l_{\{e_j\}_{j=1}^{n_e}},l_2}(1)+1} \times S^{m_{l_{\{e_j\}_{j=1}^{n_e}},l_2}(2)}$ (resp. $S^{m_{l_{\{e_j\}_{j=1}^{n_e}}}-1}$).}
	\label{fig:4}
\end{figure}

We can also show the following by applying several arguments in our proof of Main Theorem \ref{mthm:2} for example.

\begin{MainThm}
\label{mthm:3}
Suppose that a situation as in Proposition \ref{prop:1}, in Main Theorem \ref{mthm:1} and Theorems \ref{thm:1} and \ref{thm:2} is given as follows. Let $n=2$, $l_2:=l_{1,2} \geq 1$, suppose that the family $\{S_j:=S_{1,j}\}_{j=1}^{l_{1,1}} \subset {\mathbb{R}}^2$ {\rm (}$l_1:=l_{1,1}${\rm )} is a family of circles which may intersect mutually, that the relation $\overline{D_1}-D_1 \subset {\bigcup}_{j=1}^{l_{1,1}} S_j$ {\rm (}$D:=D_1${\rm )} holds and that the intersection $(\overline{D_1}-D_1) \bigcap S_j$ is non-empty for each $1 \leq j \leq l_{1,1}$. We also assume that circles in the family do not contain points which are vertical or horizontal poles of other circles in the family. In addition the functions $m_{l_{1,1},l_{1,2}}$ and $m_{l_{1,2}}$ are given in such a way that we can define the moment-like map $f=f_{(D, \{S_{j}\}_{j=1}^{l_1},m_{l_1,l_2},m_{l_2})}:=f_{(D_1, \{S_{1,j}\}_{j=1}^{l_{1,1}},m_{l_{1,1},l_{1,2}},m_{l_{1,2}})}:M_{(D, \{S_{j}\}_{j=1}^{l_1},m_{l_1,l_2},m_{l_2})}:=M_{(D_1, \{S_{1,j}\}_{j=1}^{l_{1,1}},m_{l_{1,1},l_{1,2}},m_{l_{1,2}})} \rightarrow {\mathbb{R}}^2$. As an additional condition, we also define the function $m_{l_{1,2}}$ as a function whose values are always positive integers.

Furthermore, for the Reeb graph $W_{f_{0,(D_1, \{S_{1,j}\}_{j=1}^{l_{1,1}},m_{l_{1,1},l_{1,2}},m_{l_{1,2}})}}$ of the function $f_{0,(D_1, \{S_{1,j}\}_{j=1}^{l_{1,1}},m_{l_{1,1},l_{1,2}},m_{l_{1,2}})}$, let $n_e$ denote the number of the edges of $W_{f_{0,(D_1, \{S_{1,j}\}_{j=1}^{l_{1,1}},m_{l_{1,1},l_{1,2}},m_{l_{1,2}})}}$ and $\{e_j\}_{j=1}^{n_e}$ the family of all edges of the graph.

Then, for each edge $e_i$, any function $m_{e_i,l_{1,2}}:{\mathbb{N}}_{l_{1,2+1}} \rightarrow \mathbb{N}$ whose values are always positive and whose restriction to ${\mathbb{N}}_{l_{1,2}}$ is the original function $m_{l_{1,2}}$, and the integer $m_{e_i}:={\Sigma}_{j=1}^{l_{1,2}+1} (m_{e_i,l_{1,2}}(j))+2$, there exist a non-singular real algebraic manifold $M_{e_i}$ which is also a closed and connected manifold, which is the zero set of a real polynomial map, and whose dimension is $m_{e_i}$, and a smooth real algebraic map $f_{e_i}:M_{e_i} \rightarrow {\mathbb{R}}^2$ with the following properties. 
\begin{enumerate}
\item \label{mthm:3.1} The map $f:=f_{e_i}$ is the moment-like map reconstructed from some data $(D,\{S_j\}_{j=1}^{l_1},m_{l_1,l_2},m_{l_2}):=(D_{e_i},\{S_{e_i,j}\}_{j=1}^{l_1},m_{e_i,l_1,l_2},m_{e_i,l_{2}})$ by choosing each $S_{e_i,j}$ suitably. The following conditions are satisfied.
\begin{enumerate}

\item The non-empty open set $D:=D_{e_i}$ is a connected proper subset of $D_{1}$. The intersection $(\overline{D}-D) \bigcap S_j=(\overline{D_{e_i}}-D_{e_i}) \bigcap S_j$ is non-empty for each $1 \leq j \leq l_{1,1}$.
\item The circles are defined by $S_j:=S_{1,j}$ for $1 \leq j \leq l_{1,1}$.
\item The integers $l_1$ and $l_2$ are defined by $l_1:=l_{1,1}+1$ and  $l_2:=l_{1,2}+1$. 
\item The set $S_{l_1} \bigcap S_{1,j}$ is a non-empty set for a unique integer $1 \leq j \leq l_{1,1}$. The non-empty set is a discrete two-point set.
\end{enumerate}

\item \label{mthm:3.2} The function $f_{0,e_i}:={\pi}_{2,1} \circ f_{e_i}$ is a Morse-Bott function.
\item \label{mthm:3.3} The Reeb graph $W_{f_{0,e_i}}$ is isomorphic to the graph obtained in the following way.
\begin{enumerate}
\item We choose the edge $e_i$ of $W_{f_{0,(D_1, \{S_{1,j}\}_{j=1}^{l_{1,1}},m_{l_{1,1},l_{1,2}},m_{l_{1,2}})}}$.We add two distinct new vertices $v_{e_i,1}$ and $v_{e_i,2}$ in the interior of $e_i$.
\item We add an edge to the previous new graph which contains $v_{e_i,1}$ or $v_{e_i,2}$ as a vertex and another newly added vertex. This graph is our new graph collapsing to the original Reeb graph $W_{f_{0,(D_1, \{S_{1,j}\}_{j=1}^{l_{1,1}},m_{l_{1,1},l_{1,2}},m_{l_{1,2}})}}$. We remember that the number of edges is greater than that of the original Reeb graph by $3$ and that the number of vertices is greater than that of the original graph by $3$.
\end{enumerate} 
\end{enumerate}
\end{MainThm}
\begin{proof}[A sketch of a proof]
Most of important arguments is discussed in the proof of Main Theorem \ref{mthm:2}. We only present a sketch of a proof.

We can choose circles $S_{1,j}$ and the unique additional circle $S_{l_1}=S_{l_{1,1}+1}$ as in the proof of Main Theorem \ref{mthm:2}. 
We can abuse most of the notation. We can have our map $f=f_{e_i}:M_{e_i} \rightarrow {\mathbb{R}}^2$ like one presented through (\ref{mthm:2.5})--(\ref{mthm:2.9}).

We present some remarks on the definition of our map. The function $m_{e_i,l_{1},l_{2}}$ is defined in such a way that the restriction to ${\mathbb{N}}_{l_{1,1}}$ is same as $m_{l_{1,1},l_{1,2}}$ with the set of the target being ${\mathbb{N}}_{l_2}={\mathbb{N}}_{l_{1,2}+1}$ and that $m_{e_i,l_{1},l_{2}}(l_{1,1}+1)=l_{1,2}+1=l_2$. We remember that for the function $m_{e_i,l_{2}}$, the restriction to ${\mathbb{N}}_{l_{1,2}}$ and $m_{l_{1,2}}$ agree and that the value at $l_{1,2}+1$ is $m_{e_i,l_{2}}(l_{1,2}+1):=m_{e_i}-({\Sigma}_{j=1}^{l_{1,2}} (m_{e_i,l_{2}}(j))+2)>0$.
The map $l_{\{e_j\}_{j=1}^{n_e}}:{\mathbb{N}}_{n_e} \rightarrow {\mathbb{N}} \sqcup \{0\}$, discussed in our proof of Main Theorem \ref{mthm:2}, can be abused in a natural way. It is defined as a map whose values are $1$ at the edge $e_i$ and $0$ at the remaining edges, in our new situation.

We investigate the preimage ${f_{0,e_i}}^{-1}(t)$ containing no singular points of the function $f_{0,e_i}$.  We also abuse the notation here.
Similarly, we have $L_{e,t}$ as a connected component of the intersection of the image of $f_{e_i}$ and the zero set of the real polynomial $x_1-t=0$ in ${\mathbb{R}}^2$ ($(x_1,x_2) \in {\mathbb{R}}^2$). We also have two points $p_{e,1}$ and $p_{e,2}$ in the boundary of  $L_{e,t}$. Either of the following two holds.
\begin{itemize}
\item These points $p_{e,1}$ and $p_{e,2}$ are in distinct circles represented as $S_{j_1}$ and $S_{j_2}$ and the relation $m_{e_i,l_1,l_2}(j_1)=m_{e_i,l_1,l_2}(j_2)$ holds.
The preimage ${f_{e_i}}^{-1}(L_{e,t})$ is diffeomorphic to $S^{m_{e_i,l_2}(m_{e_i,l_1,l_2}(j_1))+1} \times {\prod}_{j \in {\mathbb{N}}_{l_2}-\{m_{e_i,l_1,l_2}(j_1)\}} S^{m_{e_i,l_2}(m_{e_i,l_1,l_2}(j))}=S^{m_{e_i,l_2}(m_{e_i,l_1,l_2}(j_2))+1} \times {\prod}_{j \in {\mathbb{N}}_{l_2}-\{m_{e_i,l_1,l_2}(j_2)\}} S^{m_{e_i,l_2}(m_{e_i,l_1,l_2}(j))}$. Here "$S^{m_{e_i,l_2}(m_{e_i,l_1,l_2}(j_1))+1}=S^{m_{e_i,l_2}(m_{e_i,l_1,l_2}(j_2))+1}$" is regarded as a manifold diffeomorphic to a double of two copies of $D^{m_{e_i,l_2}(m_{e_i,l_1,l_2}(j_1))+1}=D^{m_{e_i,l_2}(m_{e_i,l_1,l_2}(j_2))+1}$.
\item These points $p_{e,1}$ and $p_{e,2}$ are in distinct circles represented as $S_{j_1}$ and $S_{j_2}$ 
and the relation $m_{e_i,l_1,l_2}(m_{e_i,l_1,l_2}(j_1)) \neq m_{e_i,l_1,l_2}(m_{e_i,l_1,l_2}(j_2))$ holds.
The preimage ${f_{e_i}}^{-1}(L_{e,t})$ is diffeomorphic to $S^{m_{e_i,l_2}(m_{e_i,l_1,l_2}(j_1))+m_{e_i,l_2}(m_{e_i,l_1,l_2}(j_2))+1} \times  {\prod}_{j \in {\mathbb{N}}_{l_2}-\{m_{e_i,l_1,l_2}(j_1),m_{e_i,l_1,l_2}(j_2)\}} S^{m_{e_i,l_2}(m_{e_i,l_1,l_2}(j))}$.  Here "$S^{m_{e_i,l_2}(m_{e_i,l_1,l_2}(j_1))+m_{e_i,l_2}(m_{e_i,l_1,l_2}(j_2))+1}$" is regarded as a manifold diffeomorphic to a manifold obtained by gluing two manifolds $D^{m_{e_i,l_2}(m_{e_i,l_1,l_2}(j_1))+1} \times \partial D^{m_{e_i,l_2}(m_{e_i,l_1,l_2}(j_2))+1}$ and $\partial (D^{m_{e_i,l_2}(m_{e_i,l_1,l_2}(j_1))+1}) \times D^{m_{e_i,l_2}(m_{e_i,l_1,l_2}(j_2))+1}$  along the boundaries in a canonical way.

\end{itemize}
 
These preimages are also connected components of the preimage ${f_{0,e_i}}^{-1}(t)$.

We can check that all properties are enjoyed for our map. 
\end{proof}
FIGUREs \ref{fig:5} and \ref{fig:6} show an explicit case for Main Theorem \ref{mthm:3}.
\begin{figure}
	
	\includegraphics[height=75mm, width=100mm]{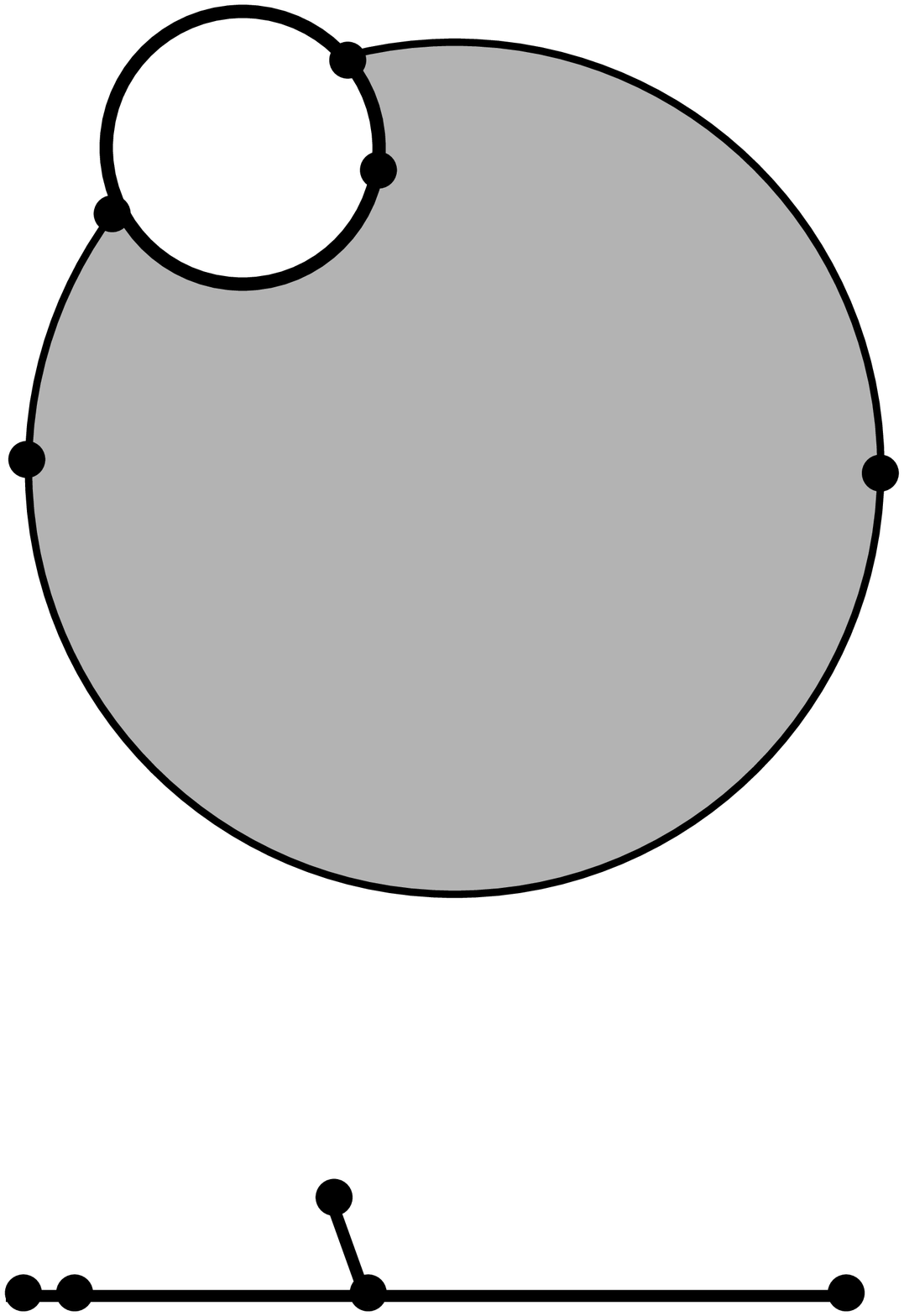}

	\caption{An example for Main Theorem \ref{mthm:3}. The image of a moment-like map $f_{(D_1,\{S_{1,j}\}_{j=1}^{l_{1,1}},m_{l_{1,1},l_{1,2}},m_{l_{1,2}})}$ into ${\mathbb{R}}^2$ for a region $D_1$ surrounded by circles and colored in gray with $n=2$ and $l_{1,1}=l_{1,2}=2$. The Reeb graph of the function $f_{0,(D_1,\{S_{1,j}\}_{j=1}^{l_{1,1}},m_{l_{1,1},l_{1,2}},m_{l_{1,2}})}:={\pi}_{2,1} \circ f_{(D_1,\{S_{1,j}\}_{j=1}^{l_{1,1}},m_{l_{1,1},l_{1,2}},m_{l_{1,2}})}$. Black dots are for singular points of the function.}
	\label{fig:5}
\end{figure}
\begin{figure}
	
	\includegraphics[height=75mm, width=100mm]{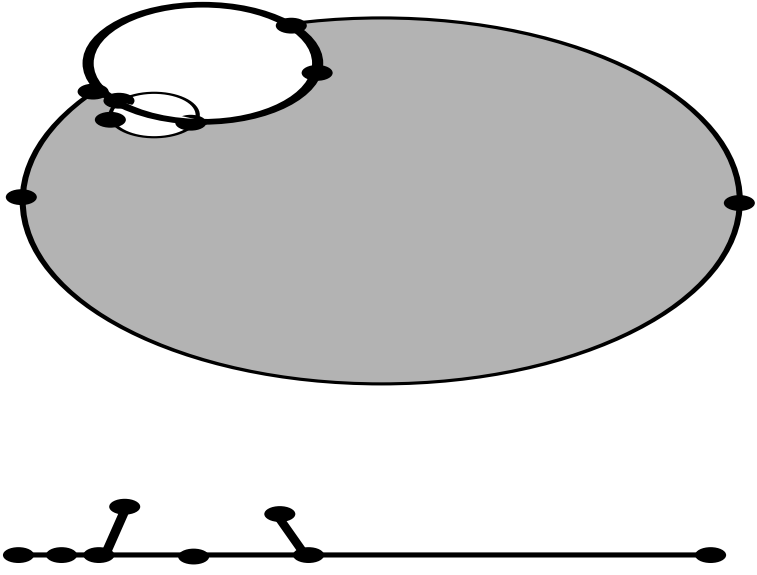}

	\caption{An example for Main Theorem \ref{mthm:3} following FIGURE \ref{fig:5}. The image of a moment-like map $f_{e_i}=f_{(D_{e_i},\{S_{e_i,j}\}_{j=1}^{l_1},m_{e_i,l_1,l_2},m_{e_i,l_2})}$ into ${\mathbb{R}}^2$ for a region $D=D_{e_i}$ surrounded by circles and colored in gray: here FIGURE \ref{fig:3} is respected and $(l_{1,1},l_1,l_2)=(2,3,l_{1,2}+1)$ ($l_{1,2}=2$). The Reeb graph of the function $f_{0,e_i}=f_{0,(D_{e_i},\{S_{e_i,j}\}_{j=1}^{l_1},m_{e_i,l_1,l_2},m_{e_i,l_2})}:={\pi}_{2,1} \circ f_{(D,\{S_j\}_{j=1}^{l_1},m_{e_i,l_1,l_2},m_{e_i,l_2})}$. Black dots are for singular points of the function.}
	\label{fig:6}
\end{figure}

	


\begin{Rem}
\label{rem:1}
FIGURE \ref{fig:7} shows an important example related to Main Theorem \ref{mthm:2}. This is a kind of counterexamples.
This is same as the case of FIGUREs \ref{fig:3} and \ref{fig:4} except the condition $m_{l_{\{e_j\}_{j=1}^{n_e}},l_2}(2)=0$.
This implies that we cannot show Main Theorem \ref{mthm:2} (\ref{mthm:2.3}) or Main Theorem \ref{mthm:3} (\ref{mthm:3.3}) under more general situations. For example, in the case the values of the function $m_{l_2}$ at some integers may be $0$. We omit examples on Main Theorem \ref{mthm:3}. 
\end{Rem}
\begin{figure}
	
	\includegraphics[height=75mm, width=100mm]{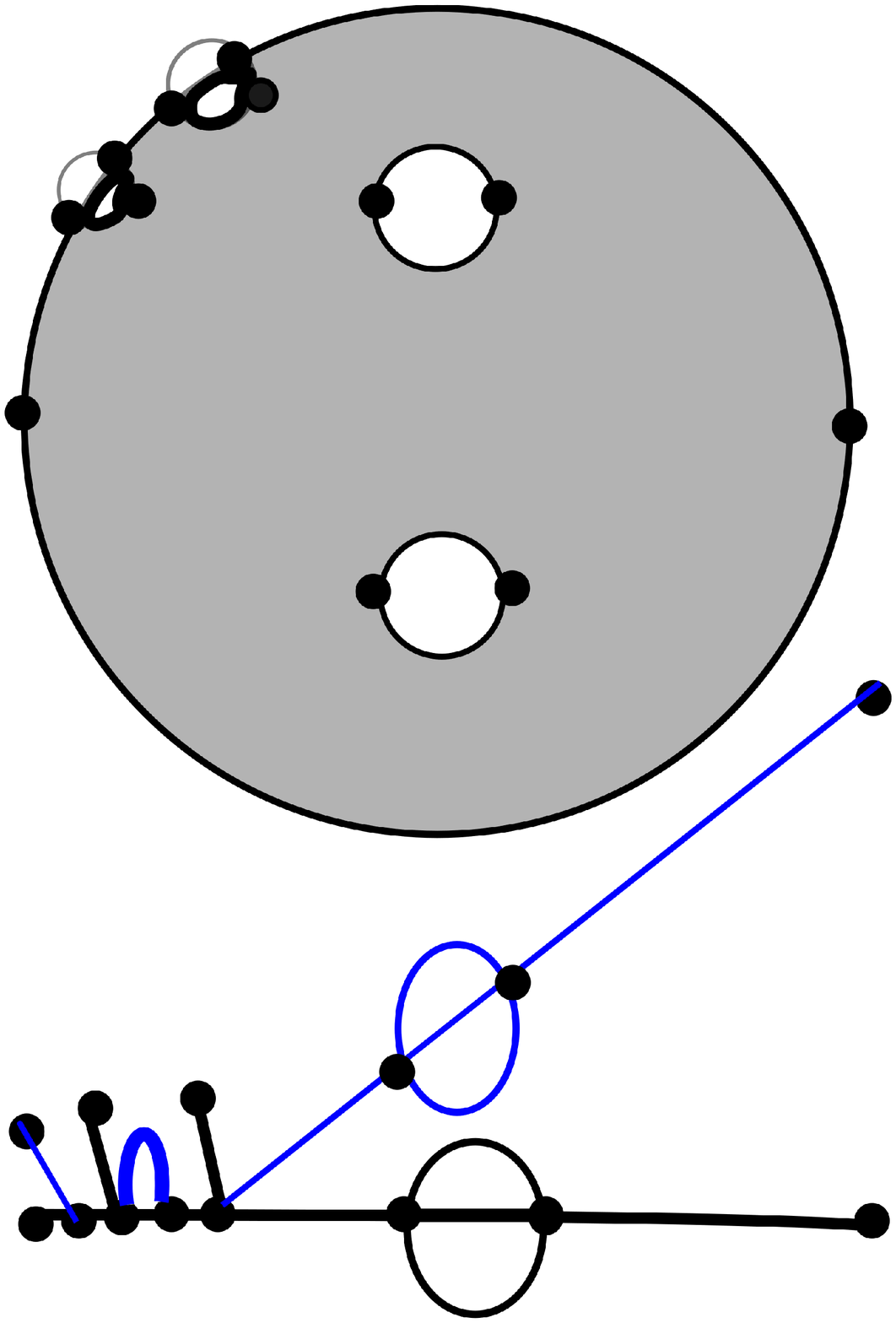}

	\caption{An example for Remark \ref{rem:1}. This follows FIGURE \ref{fig:3}. This shows a graph which is a bit different from that of FIGURE \ref{fig:4}. Here we replace a condition on the value $m_{l_{\{e_j\}_{j=1}^{n_e}},l_2}(2)$ by $m_{l_{\{e_j\}_{j=1}^{n_e}},l_2}(2)=0$, dropping the condition on the value of Main Theorem \ref{mthm:2}.}
	\label{fig:7}
\end{figure}
\section{Closing comments on our study.}
We close our paper by giving several comments.

First, we present our additional result.
\begin{MainThm}
\label{mthm:4}
Suppose that a situation and objects as in Main Theorem \ref{mthm:2} are given. We abuse the notation there.

Then, for each map $l_{\{e_j\}_{j=1}^{n_e}}:{\mathbb{N}}_{n_e} \rightarrow {\mathbb{N}} \sqcup \{0\}$ and an arbitrary integer $m_{l_{\{e_j\}_{j=1}^{n_e}}}>m_{l_{0,2}}+2$, there exist a non-singular real algebraic manifold ${M^{\prime}}_{l_{\{e_j\}_{j=1}^{n_e}}}$ which is also a closed and connected manifold, which is the zero set of a real polynomial map, and whose dimension is $m_{l_{\{e_j\}_{j=1}^{n_e}}}$, and a smooth real algebraic map ${f^{\prime}}_{l_{\{e_j\}_{j=1}^{n_e}}}:{M^{\prime}}_{l_{\{e_j\}_{j=1}^{n_e}}} \rightarrow {\mathbb{R}}^2$ with the following properties. 
\begin{enumerate}
\item \label{mthm:4.1} Each map $f:={f^{\prime}}_{l_{\{e_j\}_{j=1}^{n_e}}}$ is  the moment-like map reconstructed from some data $(D,\{S_j\}_{j=1}^{l_1},m_{l_1,l_2},m_{l_2}):=(D_{l_{\{e_j\}_{j=1}^{n_e}}},\{S_{l_{\{e_j\}_{j=1}^{n_e}},j}\}_{j=1}^{l_1},m_{l_{\{e_j\}_{j=1}^{n_e}},l_1,l_2},m_{l_{\{e_j\}_{j=1}^{n_e}},l_2})$ by choosing each $S_{l_{\{e_j\}_{j=1}^{n_e}},j}$. The following conditions are satisfied.
\begin{enumerate}
\item The non-empty open set $D:=D_{l_{\{e_j\}_{j=1}^{n_e}}}$ is a connected proper subset of $D_0$.
The intersection $(\overline{D}-D) \bigcap S_j=(\overline{D_{l_{\{e_j\}_{j=1}^{n_e}}}}-D_{l_{\{e_j\}_{j=1}^{n_e}}}) \bigcap S_j$ is non-empty for each $1 \leq j \leq l_{1,1}$.
\item We define the two integers $l_1:=l_{0,1}+{\Sigma}_{j^{\prime}=1}^{n_e} l_{\{e_j\}_{j=1}^{n_e}}(j^{\prime})$ and $l_2:=l_{0,2}+1=1+1=2$.
\item Each $S_j$ is a circle.  For $1 \leq j \leq l_{0,1}$, $S_j:=S_{j,0}$.
\item Distinct circles in the family $\{S_{l_{0,1}+j}\}_{j=1}^{l_1-l_{0,1}}$ are disjoint. For each circle $S_{l_{0,1}+j_2}$ {\rm (}$1 \leq j_2 \leq l_1-l_{0,1}${\rm )} in this family, there exists a unique circle $S_{j_1,0}$ in the family $\{S_{j,0}\}_{j=1}^{l_{0,1}}$ such that the intersection $S_{j_1,0} \bigcap S_{l_{0,1}+j_2}$ is not empty. Furthermore, the intersection is a discrete two-point set. 
\end{enumerate}
Hereafter, let ${f^{\prime}}_{0,l_{\{e_j\}_{j=1}^{n_e}}}:={\pi}_{2,1} \circ {f^{\prime}}_{l_{\{e_j\}_{j=1}^{n_e}}}$.

\item \label{mthm:4.2} Each function ${f^{\prime}}_{0,l_{\{e_j\}_{j=1}^{n_e}}}$ is a Morse-Bott function.
\item \label{mthm:4.3} Each Reeb graph $W_{{f^{\prime}}_{0,l_{\{e_j\}_{j=1}^{n_e}}}}$ is homeomorphic to the original Reeb graph $W_{f_{0,(D_0, \{S_{0,j}\}_{j=1}^{l_{0,1}},m_{l_{0,1},l_{0,2}},m_{l_{0,2}})}}$. 
\item \label{mthm:4.4} The number of vertices of the new graph $W_{{f^{\prime}}_{0,l_{\{e_j\}_{j=1}^{n_e}}}}$ is greater than that of the original graph $W_{f_{0,(D_0, \{S_{0,j}\}_{j=1}^{l_{0,1}},m_{l_{0,1},l_{0,2}},m_{l_{0,2}})}}$ by ${\Sigma}_{j=1}^{n_e} 2l_{\{e_j\}_{j=1}^{n_e}}(j)$. The number of edges of the new graph $W_{{f^{\prime}}_{0,l_{\{e_j\}_{j=1}^{n_e}}}}$ is greater than that of the original graph $W_{f_{0,(D_0, \{S_{0,j}\}_{j=1}^{l_{0,1}},m_{l_{0,1},l_{0,2}},m_{l_{0,2}})}}$ by ${\Sigma}_{j=1}^{n_e} 2l_{\{e_j\}_{j=1}^{n_e}}(j)$.

\end{enumerate}
\end{MainThm}
\begin{MainThm}
\label{mthm:5}
Suppose that a situation and objects as in Main Theorem \ref{mthm:3} are given. We abuse the notation there.

Then for each edge $e_i$, any function $m_{e_i,l_{1,2}}:{\mathbb{N}}_{l_{1,2+1}} \rightarrow \mathbb{N}$ whose values are always positive and whose restriction to ${\mathbb{N}}_{l_{1,2}}$ is the original function $m_{l_{1,2}}$, and the integer $m_{e_i}:={\Sigma}_{j=1}^{l_{1,2}+1} (m_{e_i,l_{1,2}}(j))+2$, there exist a non-singular real algebraic manifold ${M^{\prime}}_{e_i}$ which is also a closed and connected manifold, which is the zero set of a real polynomial map, and whose dimension is $m_{e_i}$, and a smooth real algebraic map ${f^{\prime}}_{e_i}:{M^{\prime}}_{e_i} \rightarrow {\mathbb{R}}^2$ with the following properties. 
\begin{enumerate}
\item \label{mthm:5.1} The map $f:={f^{\prime}}_{e_i}$ is the moment-like map reconstructed from some data $(D,\{S_j\}_{j=1}^{l_1},m_{l_1,l_2},m_{l_2}):=(D_{e_i},\{S_{e_i,j}\}_{j=1}^{l_1},m_{e_i,l_1,l_2},m_{e_i,l_{2}})$ by choosing each $S_{e_i,j}$ suitably. The following conditions are satisfied.
\begin{enumerate}

\item The non-empty open set $D:=D_{e_i}$ is a connected proper subset of $D_{1}$. The intersection $(\overline{D}-D) \bigcap S_j=(\overline{D_{e_i}}-D_{e_i}) \bigcap S_j$ is non-empty for each $1 \leq j \leq l_{1,1}$.
\item The circles are defined by $S_j:=S_{1,j}$ for $1 \leq j \leq l_{1,1}$.
\item The integers $l_1$ and $l_2$ are defined by $l_1:=l_{1,1}+1$ and  $l_2:=l_{1,2}+1$. 
\item The set $S_{l_1} \bigcap S_{1,j}$ is a non-empty set for a unique integer $1 \leq j \leq l_{1,1}$. The non-empty set is a discrete two-point set.
\end{enumerate}

\item \label{mthm:5.2} The function ${f^{\prime}}_{0,e_i}:={\pi}_{2,1} \circ {f^{\prime}}_{e_i}$ is a Morse-Bott function.
\item \label{mthm:5.3} The Reeb graph $W_{{f^{\prime}}_{0,e_i}}$ is isomorphic to the graph obtained in the following way{\rm :} we choose the edge $e_i$ of $W_{f_{0,(D_1, \{S_{1,j}\}_{j=1}^{l_{1,1}},m_{l_{1,1},l_{1,2}},m_{l_{1,2}})}}$ and add exactly 2 vertices in the interior of the edge $e_i$.
\end{enumerate}
\end{MainThm}

We can show these two by the following arguments immediately. 

Instead of choosing small circles centered at points in the given circles $S_j$ ($1 \leq j \leq l_{1,1}$) as in our proof of Main Theorems \ref{mthm:2} and \ref{mthm:3} and FIGURE \ref{fig:2} (FIGURE \ref{fig:1}), we can choose a small segment connecting two points in a circle $S_j$ ($1 \leq j \leq l_{1,1}$) and choose a small circle $S_{j^{\prime}}$ ($j^{\prime}>l_{1,1}$) in such a way that for the two arcs obtained by dividing the new circle $S_{j^{\prime}}$, the shorter one can be chosen and sufficiently close to the segment. See FIGURE \ref{fig:8}. Remember that it is sufficient to consider one case for the local curvature of the existing circle $S_j$ ($1 \leq j \leq l_{1,1}$), thanks to the symmetry. 

Except this, we can argue as in our proof of Main Theorems \ref{mthm:4} and \ref{mthm:5}. 
This completes the proof.

\begin{figure}
	
	\includegraphics[height=75mm, width=100mm]{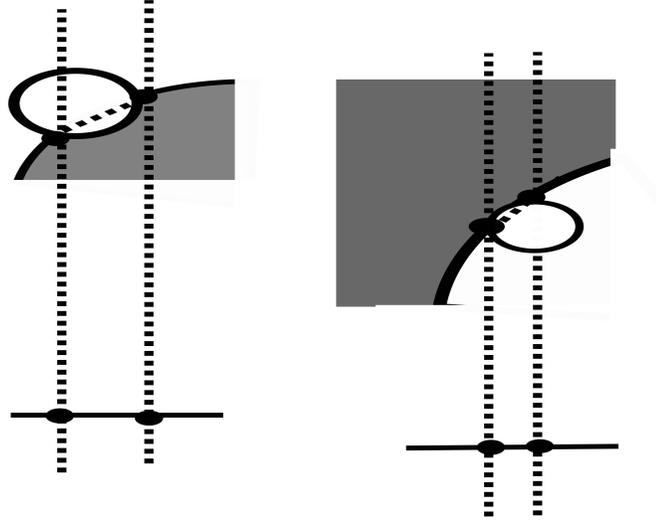}

	\caption{Local intersections of our circles. We consider one case from the four cases in FIGURE \ref{fig:1} and discuss the location and the intersection as in FIGURE \ref{fig:2}. We also show our Reeb graphs locally.}
	\label{fig:8}
\end{figure}

We present an explicit case for Main Theorem \ref{mthm:2}.

\begin{MainThm}
\label
{mthm:6}
In Main Theorem \ref{mthm:2}, let $D_0={\rm Int}\ D^2 \subset {\mathbb{R}}^2$. By applying our construction in our proof, we have the following. Note that $l_{0,1}=1$ and that $n_e=1$ for example.
\begin{enumerate}
\item In the case $l_{{e_j}_{j=1}^1}(1)={n_{e}}^{\prime} \geq 0$, we have a graph $G^{\prime}$ as follows.
\begin{enumerate}
\item The number of vertices of $G^{\prime}$ is $3{n_{e}}^{\prime}+2$.
\item The graph $G^{\prime}$ is isomorphic to a graph $G_{{\rm P},j_1,j_2}$ as follows.
\begin{enumerate}
\item Two integers $j_1$ and $j_2$ satisfy $j_1, j_2 \in \mathbb{N} \bigcup \{0\}$ and $j_1+j_2={n_e}^{\prime}$.  
\item A graph $G_{{\rm P},0,j_1,j_2}$ with exactly $2{n_e}^{\prime}+2$ vertices, labelled by the set ${\mathbb{N}}_{2{n_e}^{\prime}+2}$, and exactly $2{n_e}^{\prime}+1$ edges, the $j$-th edge of which connectes the vertices labelled by "$j$" and "$j+1$", is given.
\item The graph $G_{{\rm P},j_1,j_2}$ is obtained by adding to the previous graph $G_{{\rm P},0,j_1,j_2}$ exactly ${n_e}^{\prime}$ new vertices and exactly ${n_e}^{\prime}$ new edges. In $j_1$ vertices of these new
vertices, the ${j_1}^{\prime}$-th vertex and the vertex labelled by "$2{j_1}^{\prime}+1$" are connected by a new edge. In the remaining $j_2$ vertices of them, the ${j_2}^{\prime}$-th vertex and the vertex labelled by "$2{j_1}+2{j_2}^{\prime}$" are connected by a new edge.
\end{enumerate}
\end{enumerate}
\item Conversely, we can have a graph isomorphic to $G_{{\rm P},j_1,j_2}$ for any pair $j_1, j_2 \in \mathbb{N} \bigcup \{0\}$ satisfying $j_1+j_2={n_e}^{\prime}$ in this way.
\end{enumerate}
\end{MainThm}
\begin{proof}
We apply our construction of Main Theorem \ref{mthm:2}. We can choose 
new small ${n_{e}}^{\prime}$ circles $S_{j^{\prime}+1}$ ($1 \leq j^{\prime} \leq {n_{e}}^{\prime}$) in such a way that $j_1$ circles are centered at points of the form $(t_1,t_2)$ ($t_1<0$), and that $j_2$ circles are centered at points of the form $(t_1,t_2)$ ($t_1>0$). After that, we apply Main Theorem \ref{mthm:2} to obtain our desired graph.

This argument does not depend on the sign of $t_2$ or  whether the relation $t_2<0$ or $t_2>0$ holds. 
Remember FIGUREs \ref{fig:1} and \ref{fig:2} and related arguments.

This completes the proof all.
	

\end{proof}

We give remarks and related examples.  
\begin{Rem}
\label{rem:2}
Main Theorems \ref{mthm:2}, \ref{mthm:3}, \ref{mthm:4}, \ref{mthm:5}, and \ref{mthm:6} are regarded as our additional result to main results of \cite{kitazawa3, kitazawa7}. 
We have obtained new real algebraic functions with their Reeb graphs explicitly. 

Moreover, these studies are originally motivated by Sharko's question in \cite{sharko}. It asks whether we can have an explicit nice smooth function whose Reeb graph is isomorphic to a given graph on some closed surface or more generally, some smooth closed manifold. More precisely, these studies are also motivated by a revised question by the author. 
It asks whether we can respect singularities of the functions and preimages in addition. For related studies, see \cite{kitazawa1, kitazawa2} as pioneering studies by the author for example. \cite{saeki2} is also a related study based on our informal discussions on \cite{kitazawa1}. For the case of Morse functions such that connected components of preimages with no singular points are spheres, \cite{michalak1} is a pioneering work. \cite{michalak2} is also a related work studying fundamental deformations of Reeb graphs of Morse functions on fixed manifolds. These studies of Michalak have also motivated the author. 

We go back to exposition on our real algebraic functions and their Reeb graphs. \cite{kitazawa3} is a pioneering study related to Main Theorems \ref{mthm:2}, \ref{mthm:3}, \ref{mthm:4} and \ref{mthm:5}. They study cases where hypersurfaces $S_j$ are mutually disjoint. The study is also regarded as a study for the case $D_G$ in Remark \ref{rem:4}, presented later.

It is also remarkable that we construct smooth functions which are not real algebraic or real analytic functions in several scenes in the study. Important functions in \cite{saeki2} with \cite{kitazawa2} show explicit cases.
\end{Rem}
%

\begin{Rem}
\label{rem:3}
We may obtain revised versions of our main result under weaker conditions, for example. However, we respect simple or explicit cases. For example, we respect very explicit cases with the hypersurfaces $S_j$ being spheres of fixed radii in Main Theorem \ref{mthm:1} and show Main Theorems \ref{mthm:2}, \ref{mthm:3}, \ref{mthm:4}, \ref{mthm:5} and \ref{mthm:6}. 
\end{Rem}

Related to Remark \ref{rem:3}, we give several small examples of extensions of Main Theorem \ref{mthm:1}.

\begin{Ex}
\label{ex:1}
In Main Theorem \ref{mthm:1}, we can drop conditions as follows.
\begin{enumerate}
\item We can argue similarly in the case $\overline{D}$ may not be compact. In such a case, $M$ may not be compact.
We can argue similarly in the case $\overline{D}$ may not be connected. In such a case, $M$ may not be connected.
\item We can argue similarly in the case the hypersurface $S_j$ may not be connected and may be a union of connected components of the zero set of a polynomial $f_j$.
\item We can also argue similarly in the following case.
\begin{itemize}
\item The zero set of the polynomial $f_j$ has connected components other than (connected components of) $S_j$.
\item The connected components of the zero set of the polynomial $f_j$ except (ones of) $S_j$ are disjoint from the closure $\overline{D}$ of $D$.
\end{itemize} 
\end{enumerate}
\end{Ex}
Our Main Theorems and arguments following Main Theorem \ref{mthm:1} respect the original conditions.
\begin{Rem}
\label{rem:4}
We review some important arguments from studies \cite{kitazawa3, kitazawa9} of the author. We also respect a study \cite{bodinpopescupampusorea}, having motivated the author to present these studies. Prepare a finite and connected graph $G$ enjoying the following properties.
\begin{itemize}
\item The degree of each vertex of $G$ is $1$ or $3$.
\item The graph $G$ admits a piecewise smooth embedding $e_G:G \rightarrow {\mathbb{R}}^2$ enjoying the following properties.

\begin{itemize}
\item The function ${\pi}_{2,1} \circ e_G:G \rightarrow \mathbb{R}$ is injective at each edge $e$ of $G$. 
\item If the function ${\pi}_{2,1} \circ e_G:G \rightarrow \mathbb{R}$ has an extremum at a point $p \in G$, then $p$ is a vertex of $G$ whose degree is $1$.
\item At distinct vertices, the values of the functions ${\pi}_{2,1} \circ e_G:G \rightarrow \mathbb{R}$ are always distinct.
\end{itemize}

\end{itemize}
Then we have a natural moment-like map $f_{(D_G,\{S_{G,j}\}_{j=1}^{l_1},m_{G,l_1,l_2},m_{G,l_2})}:M_{(D_G,\{S_{G,j}\}_{j=1}^{l_1},m_{G,l_1,l_2},m_{G,l_2})} \rightarrow {\mathbb{R}}^2$ with the following properties.
\begin{enumerate}
\item We can know the existence of the data $(D_G,\{S_{G,j}\}_{j=1}^{l_1},m_{G,l_1,l_2},m_{G,l_2})$ such that the non-singular real algebraic curves $S_{G,j}$ are mutually disjoint and that the graph $e_G(G)$ is regarded as a graph enjoying the following properties. We also note that the curve $S_{G,j}$ is abused in the data as a suitably defined real polynomial $f_{G,j}$. 
The open set $D_G$ and the family $\{S_{G,j}\}_{j=1}^{l_1}$ of non-singular curves in ${\mathbb{R}}^2$ are obtained from theory \cite{bodinpopescupampusorea, sorea1}: the hypersurfaces $S_{G,j}$ are obtained by approximation of Weierstrass type respecting derivatives, for example.
\begin{itemize}
	\item The boundary $\overline{D_G}-D_G$ of the closure $\overline{D_G}$ and the disjoint union ${\sqcup}_{j=1}^{l_1} S_{G,j}$ coincide.
	\item The Reeb space of the restriction of ${\pi}_{2,1}$ to the closure $\overline{D_G}$ is homeomorphic to the graph $e_G(G)$ and the graph $e_G(G)$ is also embedded in $D_G$. In terms of \cite{bodinpopescupampusorea, sorea1}, the graph $e_G(G)$ is a {\it Poincar\'e-Reeb} graph of $D_G$.
	\item The Reeb space of the restriction of ${\pi}_{2,1}$ to the closure $\overline{D_G}$ is also regarded as a graph. The vertex set of it consists of all points containing points in the boundary $\overline{D_G}-D_G$ being also singular points of the restriction to this boundary of ${\pi}_{2,1}$. The graph $e_G(G)$ is also isomorphic to this graph.
\end{itemize} 
\item The Reeb graph $W_{f_{0,(D_G,\{S_{G,j}\}_{j=1}^{l_1},m_{G,l_1,l_2},m_{G,l_2})}}=W_{{\pi}_{2,1} \circ f_{(D_G,\{S_{G,j}\}_{j=1}^{l_1},m_{G,l_1,l_2},m_{G,l_2})}}$ of the function ${\pi}_{2,1} \circ f_{(D_G,\{S_{G,j}\}_{j=1}^{l_1},m_{G,l_1,l_2},m_{G,l_2})}= f_{0,(D_G,\{S_{G,j}\}_{j=1}^{l_1},m_{G,l_1,l_2},m_{G,l_2})}$ is isomorphic to $G$. We can have this function as a Morse-Bott function. We can show that this is a Morse-Bott function by applying (a slight generalization of) Proposition \ref{prop:1}.
As a specific case, in the case $l_2=1$ or equivalently, the case $m_{G,l_1,l_2}$ is a constant function, we can also show that this is a Morse function.
\end{enumerate}

\end{Rem}
We give an example related to Remark \ref{rem:4}.\begin{Ex}
\label{ex:2}
We first give an explicit graph $G$ and its embedding $e_G$ in Remark \ref{rem:4}. This is shown in blue in FIGURE \ref{fig:9}. FIGURE \ref{fig:9} shows the image of a moment-like map $f_{(D_{G_{\rm S}},\{S_{G_{{\rm S},j}}\}_{j=1}^{l_1},m_{G_{{\rm S},j},l_1,l_2},m_{G_{{\rm S},j},l_2})}$ into ${\mathbb{R}}^2$ reconstructed from some data $(D_{G_{\rm S}},\{S_{G_{{\rm S},j}}\}_{j=1}^{l_1},m_{G_{{\rm S}},l_1,l_2},m_{G_{{\rm S}},l_2})$. Here the region $D_{G_{\rm S}}$ is a connected open set of ${\mathbb{R}}^2$ surrounded by three circles ($l_1=3$). The Reeb graph of the function $f_{0,(D_{G_{\rm S}},\{S_{G_{{\rm S},j}}\}_{j=1}^{l_1},m_{G_{{\rm S},j},l_1,l_2},m_{G_{{\rm S},j},l_2})}$ is isomorphic to the graph $G$.

For a method to find a region $D_G$ as in Remark \ref{rem:4} whose Poincar\'e-Reeb graph is $G$ ($e_G(G)$), see the original article \cite{bodinpopescupampusorea}.

Note also that this shows an example which is not from Main Theorem \ref{mthm:2} or \ref{mthm:3}: see the four points in the center for singular points of the function $f_{0,(D,\{S_j\}_{j=1}^{l_1},m_{l_1,l_2},m_{l_2})}$ and the smallest circle containing the four points for example, and compare this case to cases of Main Theorems \ref{mthm:2} and \ref{mthm:3}. 
\end{Ex}
\begin{figure}
	
	\includegraphics[height=75mm, width=100mm]{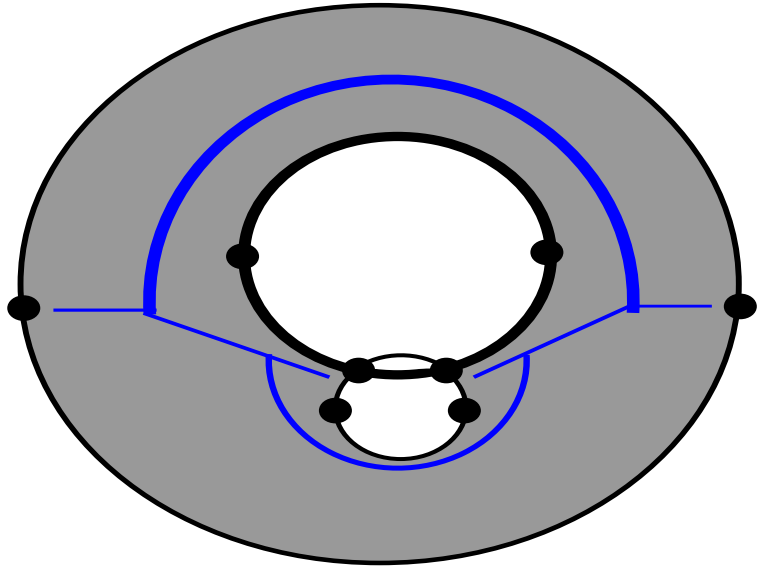}

	\caption{Example \ref{ex:2}. The image of a moment-like map $f_{(D_{G_{\rm S}},\{S_{G_{{\rm S},j}}\}_{j=1}^{l_1},m_{G_{{\rm S}},l_1,l_2},m_{G_{{\rm S}},l_2})}$ into ${\mathbb{R}}^2$ reconstructed from the data $(D_{G_{\rm S}},\{S_{G_{{\rm S},j}}\}_{j=1}^{l_1},m_{G_{{\rm S}},l_1,l_2},m_{G_{{\rm S}},l_2})$ where a region $D_{G_{\rm S}}$ is surrounded by three circles and colored in gray. The Reeb graph of the function $f_{0,(D_{G_{\rm S}},\{S_{G_{{\rm S},j}}\}_{j=1}^{l_1},m_{G_{{\rm S}},l_1,l_2},m_{G_{{\rm S}},l_2})}:={\pi}_{2,1} \circ f_{(D_{G_{\rm S}},\{S_{G_{{\rm S},j}}\}_{j=1}^{l_1},m_{G_{{\rm S}},l_1,l_2},m_{G_{{\rm S}},l_2})}$, isomorphic to the blue colored graph $G$ ($e_G(G)$). Black dots are for singular points of the function. Note that this also shows a case which is not for Main Theorem \ref{mthm:2} or Main Theorem \ref{mthm:3}. See the four dots in the center and the smallest circle containing the four points for example.}
	\label{fig:9}
\end{figure}
We present a related problem.
\begin{Prob}
Given a graph $G$ (with an embedding $e_G$) in Remark \ref{rem:4}. 
\begin{enumerate}

\item Can we define meaningful classes of  natural moment-like maps $f_G$ such that the Reeb graphs of the functions $f_{0,G}:={\pi}_{2,1} \circ  f_G$ are isomorphic to the graphs $G$ and $e_G(G)$? Can we classify such maps of certain classes we define?
\item What is the most natural moment-like map $f_G$ such that the Reeb graph of $f_{0,G}={\pi}_{2,1} \circ f_G$ is isomorphic to $G$ and $e_G(G)$? Can we define (or formulate problems of defining) such maps under certain conditions.
\end{enumerate}
\end{Prob}


\begin{thebibliography}{25}
	\bibitem{buchstaberpanov} V. M. Buchstaber and T. E. Panov, \textsl{Toric topology}, Mathematical Surveys and Monographs, Vol. 204, American Mathematical Society, Providence, RI, 2015.
	%

	\bibitem{bochnakcosteroy} J. Bochnak, M. Coste and M.-F. Roy, \textsl{Real algebraic geometry}, Ergebnisse der Mathematik und ihrer Grenzgebiete (3) [Results in Mathematics and Related Areas (3)], vol. 36, Springer-Verlag, Berlin, 1998. Translated from the 1987 French original; Revised by the authors.
		\bibitem{bochnakkucharz} J. Bochnak and W. Kucharz, \textsl{Algebraic approximation of mappings into spheres}, Michigan Mathematical Journal, vol. 34, no. 1, 1987.
	\bibitem{bodinpopescupampusorea} A. Bodin, P. Popescu-Pampu and M. S. Sorea, \textsl{Poincar\'e-Reeb graphs of real algebraic domains}, Revista Matem\'atica Complutense, https://link.springer.com/article/10.1007/s13163-023-00469-y, 2023, arXiv:2207.06871v2.
\bibitem{bott} R. Bott, \textsl{Nondegenerate critical manifolds}, Ann. of Math. 60 (1954), 248--261.
	\bibitem{delzant} T. Delzant, \textsl{Hamiltoniens p\'eriodiques et images convexes de l'application moment}, Bull. Soc. Math. France 116 (1988), No. 3, 315--339.
\bibitem{golubitskyguillemin} M. Golubitsky and V. Guillemin, \textsl{Stable Mappings and Their Singularities}, Graduate Texts in Mathematics (14), Springer-Verlag (1974).
\bibitem{kitazawa1} N. Kitazawa, \textsl{On Reeb graphs induced from smooth functions on $3$-dimensional closed orientable manifolds with finitely many singular values}, Topol. Methods in Nonlinear Anal. Vol. 59 No. 2B, 897--912, arXiv:1902.08841.
\bibitem{kitazawa2} N. Kitazawa, \textsl{On Reeb graphs induced from smooth functions on closed or open surfaces}, Methods of Functional Analysis and Topology Vol. 28 No. 2 (2022), 127--143, arXiv:1908.04340.
\bibitem{kitazawa3} N. Kitazawa, \textsl{Real algebraic functions on closed manifolds whose Reeb graphs are given graphs}, Methods of Functional Analysis and Topology Vol. 28 No. 4 (2022), 302--308, arXiv:2302.02339, 2023.
\bibitem{kitazawa4} N. Kitazawa, \textsl{Explicit construction of explicit real algebraic functions and real algebraic manifolds via Reeb graphs}, Algebraic and geometric methods of analysis 2023 “The book of abstracts”, 49—51, this is the abstract book of the conference "Algebraic and geometric methods of analysis 2023" and published after a short review (https://www.imath.kiev.ua/$\sim$topology/conf/agma2023/), https://imath.kiev.ua/$\sim$topology/conf/agma2023/contents/abstracts/texts/kitazawa/kitazawa.pdf, 2023.
\bibitem{kitazawa5} N. Kitazawa, \textsl{Notes on explicit special generic maps into Euclidean spaces whose dimensions are greater than $4$}, a revised version is submitted based on positive comments (major revision) by referees and editors after the first submission to a refereed journal, arXiv:2010.10078.



\bibitem{kitazawa6} N. Kitazawa, \textsl{A class of naturally generalized special generic maps}, arXiv:2212.03174.
\bibitem{kitazawa7} N. Kitazawa, \textsl{Construction of real algebraic functions with prescribed preimages}, submitted to a refereed journal as the second version based on positive comments by referees and editors, arXiv:2303.00953v3.
\bibitem{kitazawa8} N. Kitazawa, \textsl{Natural real algebraic maps of non-positive codimensions with prescribed images whose boundaries consist of non-singular real algebraic hypersurfaces satisfying transversality}, a previous version of the present paper, arXiv:2303.10723v5.
\bibitem{kitazawa9} N. Kitazawa, \textsl{Some remarks on real algebraic maps which are topologically special generic maps}, submitted to a refereed journal, arXiv:2312.10646. 
\bibitem{kitazawa10} N. Kitazawa, \textsl{A note on cohomological structures of special generic maps}, a revised version is submitted based on positive comments by referees and editors after the third submission to a refereed journal.
\bibitem{kohnpieneranestadrydellshapirosinnsoreatelen} K. Kohn, R. Piene, K. Ranestad, F. Rydell, B. Shapiro, R. Sinn, M-S. Sorea and S. Telen, \textsl{Adjoints and Canonical Forms of Polypols}, arXiv:2108.11747.
\bibitem{kollar} J. Koll\'ar, \textsl{Nash's work in algebraic geometry}, Bulletin (New Series) of the American Matematical Society (2) 54, 2017, 307--324.
\bibitem{kucharz} W. Kucharz, \textsl{Some open questions in real algebraic geometry}, Proyecciones Journal of Mathematics, Vol. 41 No. 2 (2022), Universidad Cat\'olica del Norte Antofagasta, Chile, 437--448.
\bibitem{maciasvirgospereirasaez} E. Mac\'ias-Virg\'os and M. J. Pereira-S\'aez, Height functions on compact symmetric spaces, Monatshefte f\"ur Mathematik 177 (2015), 119--140. 
\bibitem{michalak1} L. P. Michalak, \textsl{Realization of a graph as the Reeb graph of a Morse function on a manifold}. Topol. Methods in Nonlinear Anal. 52 (2) (2018), 749--762, arXiv:1805.06727.
\bibitem{michalak2} L. P. Michalak, \textsl{Combinatorial modifications of Reeb graphs and the realization problem}, Discrete Comput. Geom. 65 (2021), 1038--1060, arXiv:1811.08031.
\bibitem{moise} E. E. Moise, \textsl{Affine Structures in $3$-Manifold{\rm :} V. The Triangulation Theorem and Hauptvermutung}, Ann. of Math., Second Series, Vol. 56, No. 1 (1952), 96--114.
\bibitem{nash} J. Nash, \textsl{Real algbraic manifolds}, Ann. of Math. (2) 56 (1952), 405--421.
\bibitem{ramanujam} S. Ramanujam, \textsl{Morse theory of certain symmetric spaces}, J. Diff. Geom. 3 (1969), 213--229.
\bibitem{reeb} G. Reeb, \textsl{Sur les points singuliers d\'{}une forme de Pfaff compl\'{e}tement int\`{e}grable ou d\'{}une fonction num\'{e}rique}, Comptes Rendus
 Hebdomadaires des S\'{e}ances de I\'{}Acad\'{e}mie des Sciences 222 (1946), 847--849.
\bibitem{saeki1} O. Saeki, \textsl{Topology of special generic maps of manifolds into Euclidean spaces}, Topology Appl. 49 (1993), 265--293.
\bibitem{saeki2} O. Saeki, \textsl{Reeb spaces of smooth functions on manifolds}, International Mathematics Research Notices, maa301, Volume 2022, Issue 11, June 2022, 3740--3768, https://doi.org/10.1093/imrn/maa301, arXiv:2006.01689.
\bibitem{sharko} V. Sharko, \textsl{About Kronrod-Reeb graph of a function on a manifold}, Methods of Functional Analysis and
 Topology 12 (2006), 389--396.
\bibitem{sorea1} M. S. Sorea, \textsl{The shapes of level curves of real polynomials near strict local maxima},  Ph. D. Thesis, Universit\'e de Lille, Laboratoire Paul Painlev\'e, 2018.
\bibitem{sorea2} M. S. Sorea, \textsl{Measuring the local non-convexity of real algebraic curves}, Journal of Symbolic Computation 109 (2022), 482--509.
\bibitem{takeuchi} M. Takeuchi, \textsl{Nice functions on symmetric spaces}, Osaka. J. Mat. (2) Vol. 6 (1969), 283--289.
\bibitem{tognoli} A. Tognoli, \textsl{Su una congettura di Nash}, Ann. Scuola Norm. Sup. Pisa (3) 27 (1973), 167--185.

\end{thebibliography}
\end{document}